\documentclass[12pt,a4paper,twoside]{article}
\usepackage{a4wide,amsmath,amsbsy,amsfonts,amssymb,amsthm,mathrsfs}
\newcommand\ZZ{{\widehat{\mathbb Z}}}
\newcommand\Z{{\mathbb Z}}

\newcommand\Q{{\mathbb Q}}
\newcommand\F{{\mathbb F}}

\newcommand\mS{{\mathbb S}}

\newcommand\K{\Bbbk}

\newcommand\N{{\mathbb N}}
\newcommand\cL{{\mathscr L}}
\newcommand\hL{\widehat{\mathscr L}}

\newcommand\cP{{\mathscr P}}
\newcommand\ra{\rightarrow}

\newcommand\ilim{\varprojlim}
\newcommand\dlim{\varinjlim}

\newcommand\Sp{\operatorname{Sp}}

\newcommand\Spec{\operatorname{Spec}}
\newcommand\aut{\operatorname{Aut}}
\newcommand\out{\operatorname{Out}}

\newcommand\lra{\longrightarrow}
\newcommand\hookra{\hookrightarrow}
\newcommand\tura{\twoheadrightarrow}

\renewcommand{\hom}{\mathrm{Hom}}

\newcommand\sr{\stackrel}

\newcommand\sst{\scriptscriptstyle}
\newcommand\cGG{\check{\GG}}
\newcommand\hGG{\widehat{\GG}}

\newcommand\hP{\widehat{\Pi}}
\newcommand\hp{\widehat{\pi}}

\newcommand\tGG{\tilde{\GG}}

\newcommand\ssm{\smallsetminus}
\newcommand\ol{\overline}

\newcommand\cG{{\mathscr G}}
\newcommand\cS{{\mathscr S}}

\newcommand\GG{\Gamma}
\newcommand\ld{\lambda}
\newcommand\Ld{\Lambda}
\newcommand\wh{\widehat}

\newcommand\td{\tilde}
\newcommand\sg{\sigma}

\newcommand\gm{\gamma}

\def\co{\colon\thinspace}

\newtheorem{theorem}{Theorem}[section]
\newtheorem{corollary}[theorem]{Corollary}
\newtheorem{proposition}[theorem]{Proposition}
\newtheorem{lemma}[theorem]{Lemma}

\theoremstyle{definition}
\newtheorem{definition}[theorem]{Definition}
\newtheorem{remark}[theorem]{Remark}

\begin{document}

\title{A restricted Magnus property\\ for profinite surface groups}
\author{Marco Boggi \and Pavel Zalesskii}
\maketitle

\begin{abstract}
Magnus \cite{Mag} proved that, given two elements $x$ and $y$ of a finitely generated free group $F$
with equal normal closures $\langle x\rangle^F=\langle y\rangle^F$, then $x$ is conjugated either to $y$
or $y^{-1}$. More recently (\cite{B-K-Z} and \cite{Bogo}), this property, called the Magnus property,
has been generalized to oriented surface groups.

In this paper, we consider an analogue property for profinite surface
groups. While Magnus property, in general, does not hold in the profinite setting, it does hold in some
restricted form. In particular, for $\cS$ a class of finite groups, we prove that, if $x$ and $y$ are \emph{algebraically simple} 
(cf.\ Definition~\ref{simple}) elements of the pro-$\cS$ completion $\hP^\cS$ of an orientable surface group $\Pi$, such that,
for all $n\in\N$, there holds $\langle x^n\rangle^{\hP^\cS}=\langle y^n\rangle^{\hP^\cS}$, then $x$ is
conjugated to $y^s$ for some $s\in(\ZZ^\cS)^\ast$. As a matter of fact, a much more general property is
proved and further extended to a wider class of profinite completions.

The most important application of the theory above is a generalisation of the description 
of centralizers of profinite Dehn twists given in \cite{boggi} to profinite Dehn multitwists.
\end{abstract}

\section{Introduction}
Let $\Pi$ be an oriented surface group, that is to say the fundamental group of an oriented 
Riemann surface of finite type $S$. 

\begin{definition}\emph{A class of finite groups} (cf. Definition~3.1 in \cite{A-M}) is a full 
subcategory $\cS$ of the category of finite groups which is closed under taking subgroups, homomorphic 
images and extensions (meaning that a short exact sequence of finite groups is in $\cS$ whenever its 
exterior terms are). We always assume that $\cS$ contains a nontrivial group. 
\end{definition}

For $\cS$ a class of finite groups, \emph{the pro-$\cS$ completion $\hP^\cS$ of $\Pi$} is 
the inverse limit of the finite quotients of $\Pi$ which belong to $\cS$. The profinite group $\hP^\cS$ is also 
called \emph{a pro-$\cS$ surface group}. 

Let us give some examples. Fixed a non-empty set of primes $\Ld$, let then $\cS$ be the category of finite
$\Ld$-groups, i.e.\ finite groups whose orders are product of primes in $\Ld$. In this case, the
corresponding profinite completion is denoted by $\hP^\Ld$ and called the \emph{pro-$\Ld$ completion
of $\Pi$}. The two cases of interest are usually when $\Ld$ is the set
of all primes, in which case $\hP^\Ld$ is just the profinite completion $\hP$ of $\Pi$, and when $\Ld$
consists of only one prime $p$, in which case $\hP^\Ld$ is the pro-$p$ completion of $\Pi$ and is
denoted by $\hP^{(p)}$. Another important example of class  of finite groups 
is the class of finite solvable groups.

The main purpose of this paper is to prove
for the pro-$\cS$ surface group $\hP^\cS$ some properties which are analogous to the
Magnus property proved for $\Pi$ in \cite{B-K-Z} and \cite{Bogo}.

For a given element $x$ of a (profinite) group $G$, let us denote by $\langle x \rangle$ and
$\langle x \rangle^G$, respectively, the (closed) subgroup and the (closed) normal subgroup
generated (topologically) by $x$ in $G$.
The Magnus property for the discrete group $\Pi$ says that if, for two given elements
$x,y\in\Pi$, the normal subgroup $\langle x\rangle^{\Pi}$
equals the normal subgroup $\langle y\rangle^{\Pi}$, then $x$ is conjugated either to $y$ or $y^{-1}$.
This property cannot be transported literally to the profinite case since $\widehat \Z$ has
more units than just $\{\pm 1\}$ and so the property would fail already for $\widehat \Z$.  Moreover,
even if we take this into account, there are counterexamples to the analogue property which can be
formulated for the profinite completion $\hP^\cS$.

A counterexample for a free profinite group of finite rank is the following.
Let us denote by $M(G)$ the intersection of all maximal normal subgroups of a group $G$.
Let then $U$ be a normal subgroup of a finitely generated free profinite group $F$ such that $U/M(U)$
is a direct product of non-abelian simple groups (for instance, let $U$ be the kernel of the natural
epimorphism of $F$ onto the maximal prosolvable quotient of $F$).

By Proposition 8.3.6 in Chapter 8 of \cite{R-Z}, the subgroup $U$ is the
normal closure of an element $u\in U$ if and only if $U/M(U)$ is
the normal closure of $u\cdot M(U)$. Now, in an infinite direct product
of non-abelian simple groups, there are plenty of elements and
groups generated by them which are non-conjugate but normally
generate $U/M(U)$. For instance, any element with non-trivial
projection to every direct simple factor of $U/M(U)$ has this
property. However, having different order in some of the projections,
such elements are not conjugate.

A counterexample for the pro-$p$ case is instead the following.
Let $F=F(x,y)$ be the free pro-$p$ group in two generators. Then, the normal closure of $x$ has
generators which are not conjugated to powers of $x$. Indeed, this follows from the fact that
$\langle x\rangle^F$, modulo its Frattini subgroup, identifies with the completed group ring
$\F_p[[\langle y\rangle]]$ and this has more units than just the powers of $y$.

For this reason, the profinite analogue of Magnus property should be rather formulated saying
that, if $x,y\in\hP^\cS$ satisfy the stronger condition
$\langle x^n\rangle^{\hP^\cS}=\langle y^n\rangle^{\hP^\cS}$, for all $n$ which are a product of primes
in $\Ld_\cS$, then $x$ is conjugated to a power $y^s$, where $s$ is a unit of the standard pro-$\cS$ 
cyclic group $\ZZ^\cS$, i.e.\ the pro-$\cS$ completion of $\Z$. The latter group is more explicitly described 
as follows. Let $\Ld_\cS$ be the set of primes which occur as orders of groups in $\cS$, 
there is then a natural isomorphism $\ZZ^\cS\cong\prod_{p\in\Ld_\cS}\Z_p$. 

A first instance of the profinite analogue of Magnus property for a free profinite group of finite rank follows 
from a deep theorem of Wise (cf.\ \S \ref{free}), but with the further restriction that one of the two elements 
be abstract:

\begin{theorem}\label{wise}Let $\widehat F$ be a free profinite group of finite rank and $x,y\in \widehat F$, with $y$ an
element contained in an abstract dense free subgroup $F$ of $\wh{F}$. If, for all $n\in\N^+$,
there is a $k_n\in\N^+$ such that $x^{k_n}\in\langle y^n\rangle^{\widehat F}$, then $x$ is
conjugated to $y^s$ for some $s\in \widehat\Z$.
\end{theorem}

We do not know in which generality the Magnus property holds for profinite surface groups.
In what follows, we will restrict to the case where one of the two elements satisfies some geometric
conditions similar to those considered in \cite{B-K-Z}.

\begin{definition}\label{simple}\begin{enumerate}
\item Let us fix a presentation of $\Pi$ as the fundamental group of a Riemann surface $S$.
A subset of non-trivial elements $\sg=\{\gm_1,\ldots,\gm_h\}\subset\Pi$ is \emph{simple} if there is
a set of disjoint simple closed curves (\emph{briefly, s.c.c.'s}) $\tilde{\sg}=\{\td{\gm}_1,\ldots,\td{\gm}_h\}$
on $S$, such that they are two by two non-isotopic and $\td{\gm}_i$ belongs to the free
homotopy class of $\gm_i$ for $i=1,\ldots,h$. A s.c.c.\ on $S$ is \emph{peripheral}
if it bounds a $1$-punctured disc.

\item Let $\cS$ be a class of finite groups. \emph{A pro-$\cS$ surface group} $\hP^\cS$ is the pro-$\cS$ 
completion of a surface group $\Pi$. It is endowed with a natural monomomorphism 
with dense image $\Pi\hookra\hP^\cS$. Let $\GG$ be the group of mapping classes of self-homeomorphisms 
of $S$ fixing the base point of $\Pi$ and let $\hGG^\cS$ be its closure in $\aut(\hP^\cS)$. Then, a subset
of elements $\sg=\{\gm_1,\ldots,\gm_h\}\subset\hP^\cS$ is \emph{simple} if it is in the orbit of the image
of a simple set $\sg'\subset\Pi$ for the action of $\hGG^\cS$.
\item A subset of elements $\sg=\{\gm_1,\ldots,\gm_h\}\subset\hP^\cS$ is \emph{algebraically simple} if it is
in the $\aut(\hP^\cS)$--orbit of the image of a set $\sg'\subset\Pi$ which is simple for some presentation
of $\Pi$ as the fundamental group of a Riemann surface.
\end{enumerate}
\end{definition}

\begin{remark}\label{pro-p case}
Let $\Pi$ be a free group of rank $n$ and $\hP^\cS$ either its pro-$\Ld$ completion, for some non-empty 
set of primes $\Ld$, or its pro-solvable completion. Then, given a minimal set $\{\alpha_1,\ldots,\alpha_n\}$
of topological generators for $\hP^\Ld$, any element of this set and any product of commutators
$\prod_{i=1}^k[\alpha_i,\alpha_{n-i}]$ and of commutators and generators
of the form $\prod_{i=1}^k[\alpha_i,\alpha_{n-i}]\alpha_{i+1}\alpha_{i+2}\ldots\alpha_j$, for
$1\leq k\leq [n/2]$ and $i+1\leq j\leq n-k-1$, is algebraically simple.
\end{remark}

Part (i) of the above definition can be rephrased, group-theoretically, saying that there is a graph of
groups $\cG$, whose vertex groups are finitely generated free groups of rank at least $2$,
together with an isomorphism $\pi_1(\cG)\cong\Pi$ which identifies the set of edge groups of $\cG$ with
the set of cyclic subgroups of $\Pi$ generated by non-peripheral elements of $\sg$.

Part (iii) just says that, modulo automorphisms, the profinite group $\hP^\cS$ can be realized as a 
profinite completion of a discrete group $\Pi$ of the above type.

In this setting, we are actually going to prove (cf.\ \S \ref{geometric proof}) the following
stronger statement:

\begin{theorem}\label{magnus}Let $\sg=\{\gm_1,\ldots,\gm_h\}\subset\hP^\cS$ be an algebraically
simple subset and let us denote by $\sg_n^{\hP^\cS}$, for
$n\in\N^+$, the closed normal subgroup of $\hP^\cS$ generated by $n$-th powers of elements
of $\sg$. Let $y\in\hP^\cS$ be an element such that, for all $n$ a product of
primes in $\Ld_\cS$, there exists a $k_n\in\N^+$ with the property
that $y^{k_n}\in\sg_n^{\hP^\cS}$. Then, for some $s\in\ZZ^\cS$ and $i\in\{1,\ldots,h\}$,
the element $y$ is conjugated to the element $\gm_i^s$.
\end{theorem}

An immediate corollary of Theorem~\ref{magnus} is the restricted Magnus property for
profinite surface groups, mentioned above:

\begin{corollary}\label{restricted}Let $x\in\hP^\cS$ be an algebraically simple element and $y\in\hP^\cS$
an arbitrary element. Let us denote by $\langle x^n\rangle^{\hP^\cS}$ and $\langle y^n\rangle^{\hP^\cS}$
the closed normal subgroups of $\hP^\cS$ generated respectively by
$x^n$ and $y^n$, for $n\in\N^+$. If, for all $n$ a product of
primes in $\Ld_\cS$, there holds $\langle x^n\rangle^{\hP^\cS}=\langle y^n\rangle^{\hP^\cS}$, then, for
some $s\in(\ZZ^\cS)^\ast$, the element $y$ is conjugated to $x^s$. 
\end{corollary}

The restricted Magnus property has applications to Grothendieck-Teichm\"uller theory. In order to show this,  
we need more definitions. A simple closed curve on a hyperbolic Riemann surface $S$ is described 
by an unordered pair $\{\gm,\gm^{-1}\}$ in the set of conjugacy classes $\Pi/\!\sim$ of elements of the fundamental group $\Pi$ of $S$. 
Let $\cL$ be the set of non-peripheral simple closed curves on $S$. We then define the set of (non-peripheral) 
\emph{profinite simple closed curves} $\hL$ as the closure of $\cL$ in the profinite set of pairs $\cP_2(\hP/\!\sim)$ 
in the set of conjugacy classes $\hP/\!\sim$. In Theorem~4.2 \cite{boggi}, it was proved that the profinite set $\hL$ parameterizes the
set of \emph{profinite Dehn twists} in the procongruence Teichm\"uller group $\cGG(S)$ associated to $S$ (cf.\ Section~\ref{linearization}). 
This is the completion of the mapping class group $\GG(S)$ associated to the surface $S$ with respect to the congruence topology.

For $K$ a normal open subgroup of $\hP$, let $p_K\co S_K\to S$ be the associated covering with covering transformation group
$G_K:=\hP/K$. We can naturally associate to an element $\gm\in\hL$ a subspace $V_{K,\gm}$ of $H_1(\ol{S}_K,\Q_\ell)$,
where $\ol{S}_K$ is the closed Riemann surface obtained from $S_K$ filling in the punctures, in the following way. 
Let us also denote by $\gm\in\hP$ an element in the class of $\gm\in\hL$ and let $\gm^{\nu_K}$ be the smallest positive power of $\gm$ contained in $K$.
We let then $V_{K,\gm}$ be the subspace of $H_1(\ol{S}_K,\Q_\ell)$ generated by the $G_K$-orbit of the image of $\gm^{\nu_K}$
in $H_1(\ol{S}_K,\Q_\ell)$.

The main result of Section~\ref{linearization} (cf.\ Theorem~\ref{embedding}) is that the sets $\{V_{K,\gm}\}_{K\lhd_o\hP}$ separate elements of $\hL$. 
This is a non-trivial result even if restricted to the subset of simple closed curves $\cL$ of $\hL$. For $\cL$, this result was proved, with different 
techniques, in \cite{scc} (cf.\  Theorem~5.1 therein).

The above results are then used to generalize the description of centralizers of profinite Dehn twists in the procongruence Teichm\"uller group $\cGG(S)$,
given in \cite{boggi}, to profinite multitwists. A multitwist in the mapping class group $\GG(S)$ is the product of a set of Dehn twists along 
disjoint simple closed curves on $S$. We show, in particular, that the centralizer in $\cGG(S)$ of a multitwist of $\GG(S)$ is the closure 
of the centralizer of the same element in $\GG(S)$. The result we prove is actually stronger but we refer to Section~\ref{centralizers} 
for the precise statement. So far, this is probably the main application of the restricted Magnus property.

In Section~\ref{faith}, we give a linear version of some classical faithfulness results on Galois representations associated to projective hyperbolic 
curves over number fields which appeared in \cite{H-M} and \cite{boggi}.

\section{The proof of Theorem~\ref{wise}}\label{free}
In order to prove Theorem~\ref{wise}, we need a preliminary result of independent interest,
which follows from the work of Wise \cite{Wise-qc-h}:

\begin{theorem}\label{relator} Let $G=\langle F\mid r^n\rangle$ be a one-relator group with torsion and
let $y\in \widehat G$ be a torsion element of its profinite completion. Then, $y$ is conjugate 
in $\widehat G$ to a power of the image $\bar r$ of $r$ in $G$.
\end{theorem}

\begin{proof} The result is well-known for the discrete group $G$ (see for example \cite[Theorem 9.3]{A}. Hence it is enough to show that 
a torsion element $y\in \widehat G$ is conjugate to an element of $G$. 
Every one-relator group $G$ embeds naturally into a free product $G'=G\ast \Z$
which is an HNN extension $HNN(H,M,t)$ of a
one-relator group $H$ with shorter relator (length of the reduced word), where $M$ is a free subgroup generated by
subsets of the generators of the presentation of $G$ (cf.\ the
Magnus-Moldavanskii construction in Section~18.b
\cite{Wise-qc-h}). The hierarchy is finite, i.e.\ continuing further the splitting into such HNN-extensions, we  terminate at a virtually free
group of the form $\Z/n\ast F$, where $F$ is free. Let us use induction
on the length of the relator to prove that $y$ is conjugate to an element of $G$. 
If the length is zero, this means that $r=1$ and $G$ is free and then $\wh{G}$ is torsion free.
In this case, the claim trivially holds. 

According to Theorem~18.1 \cite{Wise-qc-h}, the
Magnus-Moldavanskii hierarchy is quasi-convex for any one-relator
group with torsion, i.e., for one-relator groups with torsion, the
subgroups $H$,$M$,$M^t$ are quasi-convex at each level of the
hierarchy. Wise showed that $G$ has a finite index subgroup $G_0$
that embeds as a quasi-convex subgroup of a right-angled Artin
group. It follows that every quasi-convex subgroup of G is a
virtual retract and is hence separable (cf.\ Theorem 7.3 \cite{HW}).

Thus the hierarchy is separable (including finite index subgroups
of the groups of the hierarchy) and so the profinite completion functor extends the hierarchy on 
$G$ to a hierarchy on $\widehat G$. In particular 
$\widehat G'=\widehat G\amalg \widehat \Z=HNN(\widehat H, \widehat M, t)$.
By Theorem~3.10 \cite{ZM}, any torsion element of a profinite HNN-extension is conjugate to an 
element of the base group. We may then assume that our torsion element $y$ is in $\widehat H$ and 
use the induction hypothesis to conclude that it is conjugated to an element of $H$ and so of $G'$. 
Since $\wh{G}$ is a free factor of $\wh{G}'$ and $y\in\wh{G}$, it follows that $y$ is actually conjugated to an 
element of $G$.
\end{proof}

Let us recall that an abstract group $G$ is \emph{good}, if the natural homomorphism
$G\ra\wh{G}$ of the group to its profinite completion induces an isomorphism on
cohomology with finite coefficients (cf.\ Exercises \S 2.6 \cite{Serre}).
From the proof of Theorem \ref{relator} and Theorem 1.4 \cite{GJZ}, it then follows:

\begin{theorem} One-relator groups with torsion are good.
\end{theorem}

\begin{proof}Let $G$ be a one-relator group with torsion.
As in the proof of Theorem \ref{relator}, we use induction on the length of the relator of $G$. We also use the notations of that proof. 
For $G=\Z/n\ast F$, the result is clear and this provides the base for the induction. The subgroup $H$ of $G'=G\ast \Z=HNN(H,M,t)$ 
satisfies the induction hypothesis and so by Proposition 3.5 \cite{GJZ} combined with the last paragraph of the preceding proof we deduce that 
HNN-extension $G'=HNN(H,M,t)$ is good. But the cohomology of a free product is the sum of cohomologies of the factors in both the abstract and 
the profinite situation. Therefore the isomorphism $H^i(\widehat G',M)\to H^i( G',M)$ restricts to the required isomorphism 
$H^i(\widehat G,M)\to H^i(G,M)$, for $i\geq 0$.\end{proof}

\begin{proof}[Proof of Theorem~\ref{wise}.] By Theorem \ref{relator}, the element  
$x^{k_n}\langle y^n\rangle^{\widehat F}/\langle y^n\rangle^{\widehat F}$ of the quotient group 
$\widehat F/\langle y^n\rangle^{\widehat F}$ is conjugated to a power of the element 
$y\langle y^n\rangle^{\widehat F}/\langle y^n\rangle^{\widehat F}$ for every $n\in\N^+$. 
The result then follows taking the inverse limit of all these quotients for $n\in\N^+$.
\end{proof}

\section{A geometric proof of Theorem~\ref{magnus}.}\label{geometric proof}
{\it An (orientation-preserving) Fuchsian group} $\Pi$ is a group which admits a presentation of the form:
$$\begin{array}{ll}
\Pi=&\langle\alpha_1, \dots \alpha_g,\beta_1,\dots,\beta_g,\,x_1,\ldots,x_d,\,y_1,\ldots,y_s|\\
\\
&x_1\ldots x_d\cdot y_1\ldots y_s\cdot\prod_{i=1}^g[\alpha_i,\beta_i];\,
x_1^{m_1},\ldots,x_d^{m_d},\mbox{ for }m_1,\ldots,m_d\in\N^+ \rangle.
\end{array}$$
The \emph{measure} $\mu(\Pi)$ of such a Fuchsian group is defined by:
$$\mu(\Pi)=2g-2+\sum_{i=1}^d(1-\dfrac{1}{m_i})+s.$$

The Fuchsian group $\Pi$ is \emph{hyperbolic} if $\mu(\Pi)>0$.
Geometrically, this can be reformulated by saying that $\Pi$ is a finitely generated non-elementary discrete group of isometries of
the hyperbolic plane. Hyperbolic Fuchsian groups arise as topological fundamental
groups of complex hyperbolic orbicurves (see, for instance, \cite{Mochi} for a definition).

The integer $g$ is the genus of an orbifold Riemann surface $S$ whose fundamental group
has the standard presentation given to the hyperbolic Fuchsian group $\Pi$.
Let us then observe that the same Fuchsian group as an abstract group can be given distinct
presentations corresponding to orbifold Riemann surfaces of different genuses.

Following \cite{Mochi}, the \emph{order} of the Fuchsian
group $\Pi$ is the least common multiple of the integers
$m_1,\ldots,m_d$, i.e.\ of the orders of the cyclic subgroups
generated by $x_1,\ldots,x_d$ in $\Pi$. A finite subgroup of $\Pi$
in the conjugacy class of the cyclic subgroup $\langle
x_i\rangle$, for $i=1,\ldots,d$, is called \emph{a decomposition group}.

Let $S$ be an orbifold Riemann surface whose fundamental group can be identified with the
Fuchsian group $\Pi$ described above. Then, to a finite index subgroup $K$ of $\Pi$,
is associated an unramified covering $S_K\ra S$. If $K$ is a normal subgroup, the covering is normal
with covering transformation group the quotient $G_K:=\Pi/K$. The orbifold Riemann surface $S_K$
is representable, i.e.\ the decomposition groups of its points are all trivial, if and only if, 
$\langle x_i\rangle^a\cap K=\{1\}$, for all $i=1,\ldots,d$ and all elements $a\in\Pi$. If $K$ is a
normal subgroup, then it is enough to ask that $\langle x_i\rangle\cap K=\{1\}$, for all $i=1,\ldots,d$.

The following is a slight generalization, in a different terminology, of Lemma~2.11 \cite{Mochi}, of which,
however, we prefer to give an independent proof:

\begin{lemma}\label{absolute}Let $\Pi$ be a hyperbolic Fuchsian group with the presentation
given above and let $\cS$ be a class of finite groups. Let
us assume that $\Pi$ contains a torsion free normal subgroup $H$
such that $\Pi/H\in\cS$. Let $\hP^\cS$ be the pro-$\cS$ completion of $\Pi$.
Then, there hold:
\begin{enumerate}
\item $\langle x_i\rangle\cap \langle x_j\rangle^h\neq\{1\}$, if and only if
$i=j$ and $h\in\langle x_j\rangle$. In particular, the subgroup
$\langle x_i\rangle$ is self-normalizing in the profinite group
$\hP^\cS$, for $i=1,\ldots,d$.
\item Every finite non-trivial subgroup $C$ of $\hP^\cS$ is contained in a decomposition group,
i.e.\ in a subgroup $\langle x_i\rangle^h$ for some $i\in\{1,\ldots,d\}$ and $h\in\hP^\cS$.
\end{enumerate}
\end{lemma}

\begin{proof}Let $S$ be an orbifold Riemann surface such that $\Pi=\pi_1(S)$ and let $S_H\ra S$ be the 
covering associated to the subgroup $H$ of $\Pi$. By hypothesis, $S_H$ is a Riemann surface.
\emph{The $\cS$--solenoid} $\mS^{\cS}$ (cf.\ \cite{scc}, for more details on this construction) is defined 
to be the inverse limit space
$\mS^{\cS}:=\varprojlim_{\sst \Pi/K\in\cS}S_K$ of the coverings $S_K\ra S$
associated to normal subgroups $K$ of $\Pi$ such that $\Pi/K\in\cS$. Let us observe that 
$\mS^{\cS}\cong\mS_H^{\cS}$, where the latter space is the inverse limit of the coverings $S_K\ra S_H$
associated to normal subgroups $K$ of $H$ such that $H/K\in\cS$.
There is then a series of natural isomorphisms:
$$H^k(\mS^{\cS},\Z/p)\cong H^k(\mS_H^{\cS},\Z/p):=\varinjlim_{\sst H/K\in\cS}H^k(S_K,\Z/p)\cong
\varinjlim_{\sst H/K\in\cS}H^k(K,\Z/p).$$

It is well known that a surface group is $p$-good for all primes $p$ (cf.\ (iii) of Lemma 5.12 in \cite{EHKZ}, for instance). 
This implies that $\varinjlim_{\sst H/K\in\cS}H^k(K,\Z/p)=0$, for $k>0$ and all primes $p\in\Ld_\cS$. It follows that, 
for all $p\in\Ld_\cS$, there hold $H^k(\mS^{\cS},\Z/p)=\{0\}$, for $k> 0$, and $H^0(\mS^{\cS},\Z/p)=\Z/p$.

There is a natural continuous action of $\hP^\cS$ on the $\cS$--solenoid $\mS^{\cS}$ and
the decompositions groups of $\hP^\cS$ naturally identify with the stabilizers of points of
$\mS^{\cS}$ in the inverse image of points of the orbifold Riemann surface
$S$ which have non-trivial isotropy groups.

In order to prove $(i)$, it is enough to show that the intersection of the stabilizers of two points
$P_1$ and $P_2$ contains a cyclic subgroup $C_p$ of prime order $p\in\Ld_\cS$ if and only if 
there holds $P_1=P_2$.

Let us observe that the solenoid $\mS^{\cS}$ can be triangulated by a simplicial profinite set in such a
way that the inverse images of the orbifold points of $S$ with non-trivial isotropy group are realized
inside the set of $0$-simplices. Therefore, it is possible to apply to the action of $C_p$ on the solenoid
$\mS^{\cS}$ the results of \cite{Scheiderer}.

By the results of \S 5 in \cite{Scheiderer}, item (b) of Theorem~10.5, Chap. VII \cite{Brown} generalizes
to profinite spaces. Therefore, since the profinite space $\mS^{\cS}$ is $p$-acyclic, for all $p\in\Ld_\cS$,
it follows that $(\mS^{\cS})^{C_p}$ is also $p$-acyclic,
where $(\mS^{\cS})^{C_p}$ is the fixed point set of the action of the $p$-group $C_p$ on the
$\cS$--solenoid $\mS^{\cS}$. In particular, $(\mS^{\cS})^{C_p}$ is connected and thus consists
of only one point. In particular, there holds $P_1=P_2$.

Let us now prove $(ii)$. Here, we basically proceed like in the proof of Lemma~2.11 \cite{Mochi}.
So, let $C$ be a finite subgroup of $\hP^\cS$. It then holds $C\in\cS$. 
Let us assume moreover that $C$ is solvable. Then, by induction on the order of $C$, we can further
assume that either:

\begin{itemize}
\item[a)] $C$ is of prime order $p\in\Ld_\cS$;
\item[b)] $C$ is an extension of a group of prime order $p\in\Ld_\cS$ by a non-trivial subgroup
$C_1\subseteq C$ which is contained in the decomposition group $A$.
\end{itemize}

If a) is satisfied, then, since the space $\mS^{\cS}$ is $p$-acyclic, there holds $(\mS^{\cS})^C\neq\emptyset$,
i.e.\ the subgroup $C$ is contained in a decomposition group of $\hP^\cS$.

If b) holds, by replacing the profinite group $\hP^\cS$ with its open subgroup $C_1\cdot \wh{H}^\cS$,
we can actually assume that $C_1$ is a decomposition group. But then, by $(i)$, it is also self-normalizing
and there holds $C_1=C$.

For $C$ any finite subgroup of $\hP^\cS$, the above arguments show that the Sylow subgroups of $C$
are cyclic. By a classical result of group theory, the group $C$ is then solvable and we are reduced to
the case already treated.
\end{proof}

\begin{remark}\label{determine}A consequence of Lemma~\ref{absolute} is that two (orientation-preserving) hyperbolic 
Fuchsian groups are isomorphic if and only if their pro-$\cS$ completions are isomorphic, where $\cS$ is a class of groups
containing their torsion subgroups (e.g.\ we can take for $\cS$ the class of solvable groups). In fact, it is not difficult to see that 
in order to reconstruct a Fuchsian group $\Pi$ of the above type, we need the following information:
\begin{enumerate}
\item the rank of the abelianization of $\Pi$;
\item whether $\Pi$ is virtually free or not;
\item the conjugacy classes of torsion subgroups of $\Pi$.
\end{enumerate}
The first two informations can be recovered from the homology of $\hP^\cS$ because the group $\Pi$ is $\cS$-good, 
since it contains a surface group $N$ as a normal finite index subgroup such that $\Pi/N\in\cS$ and surface groups are $\cS$-good. 
As for the third information, this can be recovered from $\hP^\cS$ by Lemma~\ref{absolute}. 
In particular, this provides a generalization of Theorem~1.4 of \cite{BCR} in the hyperbolic case.

\end{remark}

The next step is to generalize Lemma~\ref{absolute} to the fundamental group of a graph
of hyperbolic Fuchsian groups. More precisely, let $(\cG,Y)$ be a graph of groups such that
the vertex groups $G_v$, for every vertex $v\in v(Y)$, are hyperbolic Fuchsian groups and
the edge groups $G_e$ identify with maximal finite subgroups of the vertex
groups, for every edge $e\in e(Y)$. Then, we say that $\pi_1(\cG,Y)$ is \emph{a nodal Fuchsian group}.

Nodal Fuchsian groups can be characterized as topological fundamental groups of \emph{orbifold nodal 
Riemann surface}, i.e.\ nodally degenerate orbifold Riemann surfaces.

As above, if $X$ is an orbifold nodal Riemann surface, then its fundamental group is virtually torsion free,
if and only if $X$ admits a finite \'etale (eq.\ finite \'etale Galois) covering $Y\ra X$, where $Y$ is a 
connected nodal Riemann surface. 

\begin{definition}\label{decomposition}
\begin{enumerate}
\item Let $\cS$ be a class of finite groups, i.e.\ closed by taking subgroups, 
homomorphic images and extensions. \emph{A pro-$\cS$ nodal Fuchsian group} $\hP^\cS$ is the pro-$\cS$
completion of a nodal Fuchsian group $\Pi$.
\item We say that a finite subgroup $D$ of a nodal Fuchsian group $\pi_1(\cG,Y)$ or of its
pro-$\cS$ completion $\hp_1^\cS(\cG,Y)$ is \emph{a decomposition group of type $I$} if it is in the
conjugacy class of a decomposition subgroup $I$ of a vertex group of $(\cG,Y)$.
\end{enumerate}
\end{definition}

The following result is a result of independent interest we need in order to generalize
Lemma~\ref{absolute} to pro-$\cS$ nodal Fuchsian groups:

\begin{lemma} \label{Bass-Serre}Let $\cS$ be a class of finite groups
and let $(\cG,Y)$ be a finite graph of (discrete) groups, with  finite edge
groups, such that $G=\pi_1(\cG,Y)$ is residually-$\cS$. Then, the
pro-$\cS$ completion $\widehat{G}^\cS$ of $G$ is isomorphic to the
pro-$\cS$ fundamental group $\hp^\cS_1(\widehat{\cG}^\cS,Y)$ of the finite
graph of pro-$\cS$ groups $(\widehat{\cG}^\cS,Y)$ obtained from
$(\cG,Y)$ by taking the pro-$\cS$ completion of each  vertex group
and the vertex groups of $(\widehat{\cG}^\cS,Y)$  embed in $\widehat{G}^\cS$.
\end{lemma}

\begin{proof} Since $G$ is residually $\cS$ one can  find an open (in the pro-$\cS$ 
topology) subgroup $H$ of $G$ that intersects trivially all the edge groups.  Then it suffices
to show that the pro-$\cS$ topology of $G$  or equivalently of $H$ induces the full
pro-$\cS$ topology on $H\cap \cG(v)$. But $H\cap \cG(v)$ is a
free factor of $H$ so this statement is just Corollary 3.1.6 in \cite{R-Z}.
\end{proof}

Let us then extend Lemma~\ref{absolute} to pro-$\cS$ nodal Fuchsian groups:

\begin{lemma}\label{absolute2}Let $\Pi:=\pi_1(\cG,Y)$ be a nodal Fuchsian group and let $\cS$ be 
a class of finite groups. Let us assume that $\Pi$ contains a torsion free
normal subgroup $H$ such that $\Pi/H\in\cS$. Let then $\hP^\cS$ be
the pro-$\cS$ completion of $\Pi$. Let $D_1$ and $D_2$ be decomposition groups of $\hP^\cS$ of
type $I_1$ and $I_2$, respectively. Then, there hold:

\begin{enumerate}
\item The profinite group $\hP^\cS$ is virtually torsion free.
\item $\hP^\cS=\hp^\cS_1(\widehat{\cG}^\cS,Y)$ and the vertex groups of 
$(\widehat{\cG}^\cS,Y)$ embed in $\hP^\cS$.
\item $D_1\cap D_2\neq\{1\}$, if and only if $D_1=D_2$ and $I_1$, $I_2$ are contained and conjugated 
in a vertex group $G_v$ for some vertex $v\in v(Y)$.
\item The decomposition groups of $\hP^\cS$ are self-normalizing.
\item Every finite non-trivial subgroup $C$ of $\hP^\cS$ is contained in a decomposition group.
\end{enumerate}
\end{lemma}

\begin{proof} Since the subgroup $H$ is torsion free, it is a free product of surface groups. Moreover, 
since $\Pi/H\in\cS$, the closure of the subgroup $H$ in $\hP^\cS$ coincides
with its pro-$\cS$ completion $\wh{H}^\cS$ and so is a free pro-$\cS$ product of pro-$\cS$ surface
groups. By Theorem~3.10 and Remark~3.18 \cite{ZM}, any finite subgroup of the profinite group $\wh{H}^\cS$ 
is contained in a free factor of $\wh{H}^\cS$. As observed in the proof of Lemma~\ref{Bass-Serre}, such a 
free factor is the pro-$\cS$ completion of the corresponding free factor of the group $H$, hence it is a 
pro-$\cS$ surface group, which is torsion free. It follows that $\widehat{H}^\cS$ is torsion free. 
This proves (i).

In particular, by the above proof, the group $\Pi=\pi_1(\cG,Y)$ is residually-$\cS$. Item $(ii)$ then 
follows from Lemma~\ref{Bass-Serre}.

By Theorem~3.10 and Remark~3.18 \cite{ZM}, any finite subgroup of the profinite group $\hP^\cS$ is 
contained in a vertex group of $(\widehat{\cG}^\cS,Y)$. Moreover, by Theorem~3.12 \cite{ZM}, the intersection 
of the conjugates of two distinct vertex groups is conjugate to a subgroup of an edge group.

Since the edge groups of $(\widehat{\cG}^\cS,Y)$ identify with
decompositions groups of the vertex groups, if $D_1\cap D_2\neq\{1\}$, then they are both conjugated
to the same decomposition group of some vertex group of $(\widehat{\cG}^\cS,Y)$.
Thus, items $(iii)$, $(iv)$, $(v)$ follows from Lemma~\ref{absolute}.
\end{proof}

\begin{proof}[Proof of Theorem~\ref{magnus}.]It is not restrictive to assume that
$\sg=\{\gm_1,\ldots,\gm_h\}$ is a simple subset of $\Pi$. Let $\sg_n^\Pi$, for $n\in\N^+$, the closed
normal subgroup of $\Pi$ generated by $n$-th powers of elements of $\sg$.

Let us denote by $S_n$ the orbifold nodal Riemann surface obtained topologically from $S$ glueing a
disc to each s.c.c.\ in $\sg$ with an attaching map of degree $n$. Then, there is a natural isomorphism
$\pi_1(S_n)\cong\Pi/\sg_n^{\Pi}$.

Therefore, the quotient group $\Pi/\sg_n^{\Pi}$ is a nodal Fuchsian group
whose decomposition groups are the conjugacy classes of the subgroups generated by the
images of the elements in $\sg$ and the quotient group $\hP^\cS/\sg_n^{\hP^\cS}$ is the
pro-$\cS$ completion of $\Pi/\sg_n^{\Pi}$.

Let us prove that, if $n$ is a product of primes in $\Ld_\cS$, then the nodal Fuchsian group 
$\Pi/\sg_n^{\Pi}$ satisfies the hypotheses of Lemma~\ref{absolute2}. It is enough to show that there 
is a torsion free, normal subgroup $H$ of the quotient group $\Pi/\sg_n^{\Pi}$ of index a product 
of powers of primes in $\Ld_\cS$ and such that the quotient of $\Pi/\sg_n^{\Pi}$ by $H$ is a 
metabelian group, or, equivalently, that there exists a normal, metabelian $\Ld_\cS$-covering 
$S_n''\ra S_n$ such that $S_n''$ is representable.

Let $S'\ra S$ be the abelian $\Ld_\cS$-covering associated to the characteristic subgroup $[\Pi,\Pi]\Pi^n$ 
of $\Pi$. This covering has the property that its restriction to every s.c.c.\ of $S'$, which covers either a
non-separating s.c.c.\ or, in case there is more than one puncture on $S$, a peripheral s.c.c.\ on $S$, 
has degree $n$.

Since $\sg_n^{\Pi}<[\Pi,\Pi]\Pi^n$, there is an induced $\Ld_\cS$-covering $S_n'\ra S_n$ which ramifies
with order $n$ over the orbifold points of $S_n$ corresponding to the non-separating and, in case there
is more than one puncture on $S$, the peripheral s.c.c.'s in $\sg$. Therefore the orbifold Riemann 
surface $S_n'$ is representable over those points.

From the same argument used in the proof of Lemma~3.10 \cite{sym}, it follows that a s.c.c.\ contained
in the inverse image in $S'$ of a non-peripheral separating s.c.c.\ $\gm$ in the set $\sg$ is non-separating
and its image in $S_n'$ is homologically non-trivial. In case there is only one peripheral s.c.c.\ in $\sg$, a
s.c.c.\ contained in its inverse image in $S'$ has also homologically non-trivial image in $S_n'$.

Let $\Pi':=[\Pi,\Pi]\Pi^n$. The metabelian $\Ld_\cS$-covering $S''\ra S$ associated to the normal subgroup
$\sg_n^{\Pi}[\Pi',\Pi'](\Pi')^n$ of $\Pi$ has then the property that its restriction to every s.c.c.\ of $S''$,
lying above a s.c.c.\ of $\sg$, has degree $n$. Therefore, the induced metabelian $\Ld_\cS$-covering 
$S_n''\ra S_n$ is such that $S_n''$ is representable.

For an element $a\in \hP^\cS$, let us denote by $\bar{a}$ its image in the quotient group
$\hP^\cS/\sg_n^{\hP^\cS}$. The image $\bar{y}$ of the given $y$ then has finite order.

From Lemma~\ref{absolute2}, it follows that $\bar{y}\in\langle\bar{\gm}_i\rangle^{\bar{x}}$, for some
$i\in\{1,\ldots,h\}$ and $\bar{x}\in\hP^\cS/\sg_n^{\hP^\cS}$. Since this holds for all $n$ which are a product
of primes in $\Ld_\cS$, by an inverse limit argument, it actually holds $y\in\langle\gm_i\rangle^x$,
for some $i\in\{1,\ldots,h\}$ and $x\in\hP^\cS$.
\end{proof}

We say that an element $x$ of a profinite group $G$ is \emph{full} if the $p$-component
$\langle x\rangle^{(p)}$ of the pro-cyclic group generated by $x$ is non-trivial for every prime
$p$ dividing the order of $G$.

By cohomological methods, it is possible to show that normalizers of full elements in a 
non-abelian pro-$\cS$ surface groups are pro-cyclic. For algebraically simple elements in the 
pro-$\cS$ completion of a surface group, this also follows from an argument similar to the one given in the 
proof of Theorem~\ref{magnus}:

\begin{proposition}\label{normalizers}Let $\cS$ be a class of finite groups and let
$\hP^\cS$ be the pro-$\cS$ completion of a non-abelian surface group $\Pi$.
Let $x$ be an algebraically simple element of $\hP^\cS$. Then, for all $n\in\ZZ^\cS\ssm\{0\}$, there holds:
$$N_{\hP^\cS}(\langle x^n\rangle)=N_{\hP^\cS}(\langle x\rangle)=\langle x\rangle.$$
\end{proposition}

\begin{proof}We can assume that $x$ is a simple element of $\Pi$. The quotient group 
$\Xi_h:=\Pi/\langle x^h\rangle^\Pi$ is a nodal Fuchsian group which satisfies the hypotheses of 
Lemma~\ref{absolute2}. Let $x_h$ be the image of $x$ in $\Xi_h$. Then,
the cyclic group $C_h$ generated by $x_h$ is a decomposition group of $\Xi_h$.
From Lemma~\ref{absolute2}, it follows that, in the pro-$\cS$ completion $\wh{\Xi}_h^\cS$ of
$\Xi_h$, for $n\in\N^+$ and $h>n$, there holds:
$N_{\wh{\Xi}_h^\cS}(C_h)=N_{\wh{\Xi}_h^\cS}(C_h^n)=C_h$.
The conclusion of the proposition then follows taking the inverse limit for $h\ra\infty$.
\end{proof}

\emph{A multi-curve} (cf.\ Definition~\ref{curve-complex}) $\sg$ on a Riemann surface $S$ is a set 
$\{\gm_0,\ldots,\gm_k\}$ of disjoint, non-trivial, non-peripheral s.c.c.'s on $S$, such that they are 
two by two non-isotopic. The complement $S\ssm\sg$ is then a disjoint union of hyperbolic Riemann 
surfaces $\coprod_{i=0}^h S_i$ and, for some choices of base points and a path between them, 
the fundamental group $\Pi_i:=\pi_1(S_i)$, for $i=0,\ldots,h$, identifies with a subgroup of $\Pi:=\pi_1(S)$. 

Let $Y_\sg$ be the graph which has for vertices the connected components of $S\ssm\sg$ and for edges
the elements of $\sg$, where two vertices $S_i$ and $S_j$ are joined by the edge $\gm_l$ if $\gm_l$ is
a boundary component of both $S_i$ and $S_j$. Then, the group $\Pi$ is naturally isomorphic to the
fundamental group of the graph of groups $(\cG,Y_\sg)$ with vertex groups $\cG_i=\Pi_i$, for 
$i=0,\ldots,h$, and with edge groups the cyclic groups $C_j$, for $j=0,\ldots k$, generated by 
representatives in $\Pi$ of the s.c.c.'s in $\sg$.

From $(ii)$ Lemma~\ref{absolute2} and an argument similar to that of the proof of Theorem~\ref{magnus},
it follows a description of the pro-$\cS$ completion of $\Pi$ in terms of the above graph of groups:

\begin{theorem}\label{subsurface}Let $\sg$ be a multi-curve on a Riemann surface $S$, let 
$S\ssm\sg=\coprod_{i=0}^h S_i$ be the decomposition in connected components and let $\Pi_i$ be the 
fundamental group of $S_i$, for $i=0,\ldots,h$. The fundamental group $\Pi$ of $S$ is naturally
isomorphic to the fundamental group of the graph of groups $(\cG,Y_\sg)$, described above, with vertex 
groups the groups $\Pi_i$, for $i=0,\ldots,h$. Let $\cS$ be a class of finite groups.
Then, the pro-$\cS$ completion $\hP^\cS$ of $\Pi$ is the pro-$\cS$ fundamental group of the graph of 
profinite groups $(\widehat{\cG}^\cS,Y_\sg)$ whose vertex and edge groups are the pro-$\cS$ completions 
of vertex and edge groups of the graph of groups $(\cG,Y_\sg)$. Moreover, the vertex groups $\hP^\cS_i$,
for $i=0,\ldots,h$, embed in $\hP^\cS$ and are their own normalizers in this group.
\end{theorem}

\begin{proof}Let us also denote by $\sg$ a set of representatives in $\Pi$ of the s.c.c.'s in $\sg$.
As in the proof of Theorem~\ref{magnus}, let us consider the quotient group $\Pi/\sg_n^{\Pi}$,
which is a nodal Fuchsian group satisfying the hypotheses of Lemma~\ref{absolute2}.
This group is the fundamental group of the graph of groups
$(\cG_n,Y_\sg)$ whose vertex groups are the fundamental groups of the orbifolds obtained, from
the connected components of the manifold with boundary obtained cutting $S$ in $\sg$,
attaching discs to their boundary components with an attaching map of degree $n$. 
 
By $(ii)$ Lemma~\ref{absolute2}, its pro-$\cS$ completion $\hP^\cS/\sg_n^{\hP^\cS}$ is the pro-$\cS$ 
fundamental group $\hP^\cS_1(\widehat{\cG},Y_\sg)$ of the finite graph of pro-$\cS$ groups 
$(\widehat{\cG}^\cS_n,Y_\sg)$ obtained from $(\cG_n,Y)$ by taking the pro-$\cS$ completion of each  
vertex group of $(\cG_n,Y_\sg)$ and the vertex groups of $(\widehat{\cG}^\cS_n,Y_\sg)$  embed in 
$\hP^\cS/\sg_n^{\hP^\cS}$. 

The inverse limit, for $n\ra\infty$, of the graphs of pro-$\cS$ groups $(\widehat{\cG}^\cS_n,Y_\sg)$
is the graph of groups $(\widehat{\cG}^\cS,Y_\sg)$, whose vertex groups are the pro-$\cS$ completions
of the fundamental groups of the connected components of $S\ssm\sg$. The pro-$\cS$ fundamental
group of $(\widehat{\cG}^\cS,Y_\sg)$ is the inverse limit, for $n\ra\infty$, of the groups 
$\hP^\cS/\sg_n^{\hP^\cS}$, i.e.\ the pro-$\cS$ completion $\hP^\cS$ of $\Pi$. Therefore,
the vertex groups of $(\widehat{\cG}^\cS,Y_\sg)$ embed in $\hP^\cS$. The last statement in the theorem
then follows from Corollary~3.13 and Remark~3.18 \cite{ZM}.
\end{proof}

\section[Relative versions of the restricted Magnus property.]
{Relative versions of the restricted Magnus\\ property.}\label{relative version}

For the applications given in Section~\ref{linearization}, we need a finer result than
Theorem~\ref{magnus}. Let us give the following definition:

\begin{definition}\label{relative compl}
\begin{enumerate}
\item For a given profinite group $G$ (possibly finite) and a given prime $p\geq 2$, let us denote
respectively by $G^{(p)}$ and $G^{nil}$ its maximal pro-$p$ and pro-nilpotent quotients. It is clear
that, if the group $G$ is pro-nilpotent, there holds $G=\prod G^{(p)}$ and the pro-$p$ group $G^{(p)}$
is naturally identified with the $p$-Sylow subgroup of $G$.

\item For $K$ a normal open subgroup of a profinite group $G$ and a given prime $p\geq 2$, let
$G_K^{(p)}$ and $G_K^{nil}$ be, respectively, the quotients of $G$ by the kernels of
the natural epimorphisms $K\ra K^{(p)}$ and $K\ra K^{nil}$. Let us call them, respectively, 
the relative maximal pro-$p$ and pro-nilpotent quotients of $G$ with respect to the subgroup $K$.

\item For $K$ a normal finite index subgroup of a discrete group $G$ and a given prime $p\geq 2$,
let us denote by $\wh{G}_K^{(p)}$ and $\wh{G}_K^{nil}$, respectively, the quotients of the profinite
completion $\wh{G}$ by the kernels of the natural epimorphisms $\wh{K}\ra \wh{K}^{(p)}$ and
$\wh{K}\ra\wh{K}^{nil}$. Let us call $\wh{G}_K^{(p)}$ and $\wh{G}_K^{nil}$, respectively, the relative pro-$p$
and pro-nilpotent completion of $G$ with respect to the subgroup $K$.
\end{enumerate}
\end{definition}

Mochizuki's Lemma~\ref{absolute} admits the following generalization to relative
pro-$p$ completions of hyperbolic Fuchsian groups:

\begin{lemma}\label{relative}Let $\Pi$ be a hyperbolic Fuchsian group with the presentation given in
\S \ref{geometric proof}. Let $K$ be a finite index normal subgroup of $\Pi$ which contains a torsion
free normal subgroup $H$ of index a power of $p$, for a prime $p\geq 2$. Let then $\hP_K^{(p)}$ be
the relative pro-$p$ completion  of $\Pi$ with respect to the subgroup $K$. Let $D_i$ be a decomposition 
group of $\hP_K^{(p)}$ in the conjugacy class of the cyclic subgroup $\langle x_i\rangle$, for $i=1,\ldots,n$.
Then, there hold:

\begin{enumerate}
\item For all $x\in\hP_K^{(p)}$, there holds $D_i^{(p)}\cap (D_j^{(p)})^x\neq\{1\}$, if and only if, $i=j$ and
$D_i=(D_i)^x$. In particular, for all $i=1,\ldots,n$ such that $D_i^{(p)}\neq\{1\}$, there is a series of
identities:
$$N_{\hP_K^{(p)}}(D_i)=N_{\hP_K^{(p)}}(D_i^{(p)})=D_i.$$
\item Every finite nilpotent subgroup $C$ of $\hP_K^{(p)}$, such that $C^{(p)}\neq\{1\}$,  is contained
in a decomposition group of $\hP_K^{(p)}$.
\item Therefore, the decomposition groups $D$ of $\Pi_K^{(p)}$, such that $D^{(p)}\neq\{1\}$, may
be characterized as the maximal finite nilpotent subgroups $M$ of the profinite group $\Pi_K^{(p)}$
such that $M^{(p)}\neq\{1\}$.
\end{enumerate}
\end{lemma}

\begin{proof}Let $S$ be a hyperbolic orbifold Riemann surface such that $\Pi\cong\pi_1(S)$ and let
$S_K\ra S$ be the Galois unramified covering associated to the normal finite index subgroup $K$
of $\Pi$. Let us also denote by $S_H\ra S_K$ the normal unramified $p$-covering associated to the subgroup 
$H$ of $K$. By hypothesis, $S_H$ is a hyperbolic Riemann surface and the $p$-adic solenoid $\mS_K^{(p)}$
of $S_K$, i.e.\ the inverse limit of all unramified $p$-coverings of $S_K$, identifies with the $p$-adic
solenoid $\mS_H^{(p)}$ of $S_H$. Let $\{L\lhd_o H\}$ be the set of subgroups of $H$ open for
the pro-$p$ topology. There is then a series of natural isomorphisms:
$$H^k(\mS_K^{(p)},\Z/p)\cong H^k(\mS_H^{(p)},\Z/p):=\varinjlim_{L\lhd_o H}H^k(S_L,\Z/p)\cong
\varinjlim_{L\lhd_o  H}H^k(L,\Z/p).$$
Since $H$ is $p$-good, there hold $H^k(\mS_K^{(p)},\Z/p)=\{0\}$, 
for $k> 0$, and $H^0(\mS_K^{(p)},\Z/p)=\Z/p$.

There is a natural continuous action of $\Pi_K^{(p)}$ on the $p$-adic solenoid $S_K^{(p)}$ and
the decompositions groups of $\Pi_K^{(p)}$ identify with the stabilizers of points of
$S_K^{(p)}$ in the inverse image of points of the orbifold Riemann surface
$S$ which have non-trivial isotropy groups. 

The proof then proceeds exactly like the proof of Lemma~\ref{absolute}.
In order to prove $(i)$, it is enough to show that the intersection of the stabilizers of two such points
$P_1$ and $P_2$ contains a cyclic subgroup $C_p$ of order $p$ if and only if $P_1=P_2$.

Since the profinite space $\mS_K^{(p)}$ is $p$-acyclic, it follows that $(\mS_K^{(p)})^{C_p}$ is also 
$p$-acyclic, where $(\mS_K^{(p)})^{C_p}$ is the fixed point set of the action of the $p$-group $C_p$ on 
the $p$-adic solenoid $\mS_K^{(p)}$. In particular, $(\mS_K^{(p)})^{C_p}$ is connected and thus consists
of only one point. In particular, there holds $P_1=P_2$.

The proof of item $(ii)$ essentially follows from the same arguments of the proof of
item $(ii)$ of Lemma~\ref{absolute}.
Since $C$ is nilpotent, it has a unique normal and, by hypothesis, non-trivial
$p$-Sylow subgroup $C^{(p)}$. Proceeding by induction on the order of $C$, we can assume either
that $C$ is of order $p$ or that $C$ is an extension of a group of prime order by a non-trivial subgroup
$C_1\leq C$, which is contained in a decomposition group $D$ and such that $C^{(p)}\leq C_1$,.

In the first case, since the profinite space $\mS_K^{(p)}$ is $p$-acyclic and of finite dimension,
it follows that $(\mS_K^{(p)})^{C}$ is non-empty, i.e.\ $C$ is contained in a decomposition group.

In the second case, replacing $\hP_K^{(p)}$ by its open subgroup $\wh{H}^{(p)}\cdot C$, which is also
the relative pro-$p$ completion of a hyperbolic Fuchsian group with respect to some normal subgroup, we may assume 
that actually $C_1=D\leq C$. Since $C_1$ is normal in $C$ and $D$ is self-normalizing, it follows $C=D$.

Item $(iii)$ is just a reformulation of $(ii)$.
\end{proof}

The next step is to generalize Lemma~\ref{relative} to nodal Fuchsian groups.

\begin{definition}\label{decomposition2}
\begin{enumerate}
\item \emph{A virtual pro-$p$ nodal Fuchsian group} $\hP_K^{(p)}$ is the relative pro-$p$ completion
of a nodal Fuchsian group $\Pi$ with respect to a given finite index normal subgroup $K$.
\item We say that a finite subgroup $D$ of a virtual pro-$p$ nodal Fuchsian group
$\hp_1(\cG,Y)^{(p)}_K$ is \emph{a decomposition group of type $I$} if it is in the
conjugacy class of a decomposition subgroup $I$ of a vertex group of $(\cG,Y)$.
\item Sets of simple and algebraically simple elements of $\hP_K^{(p)}$ are defined as in
Definition~\ref{simple}. Thus, a subset of elements $\sg=\{\gm_1,\ldots,\gm_h\}\subset\hP_K^{(p)}$ is
\emph{algebraically simple} if it is in the $\aut(\hP_K^{(p)})$--orbit of the image of a set $\sg'\subset\Pi$ which
is simple for some presentation of $\Pi$ as the fundamental group of a Riemann surface.
\end{enumerate}
\end{definition}

\begin{lemma}\label{relative2}Let $\Pi:=\pi_1(G,Y)$ be a nodal Fuchsian group. Let $K$ be a finite
index normal subgroup of $\Pi$ and $p\geq 2$ a prime such that, for some normal subgroup $H$ of
$K$ of index a power of $p$ and for all vertices $v$ of $Y$, the subgroups
$H\cap G_v$ are torsion free. Let then $\hP_K^{(p)}$ be the relative pro-$p$ completion of $\Pi$
with respect to the subgroup $K$. Let $D_1$ and $D_2$ be decomposition groups of $\hP_K^{(p)}$ of
type $I_1$ and $I_2$, respectively, such that their $p$-Sylow subgroups are non-trivial. Then, there hold:

\begin{enumerate}
\item The profinite group $\hP_K^{(p)}$ is virtually torsion free.
\item It holds $D_1^{(p)}\cap D_2^{(p)}\neq\{1\}$, if and only if $D_1=D_2$ and
$I_1$, $I_2$ are contained and conjugated in a vertex group $G_v$ for some vertex $v\in v(Y)$.
In particular, if $D$ is a decomposition group  of $\hP_K^{(p)}$ such that $D^{(p)}\neq\{1\}$,
there is a series of identities:
$$N_{\hP_K^{(p)}}(D)=N_{\hP_K^{(p)}}(D^{(p)})=D.$$
\item Every finite nilpotent subgroup $C$ of $\hP_K^{(p)}$, such that $C^{(p)}\neq\{1\}$, is contained in
a decomposition group.
\item Therefore, the decomposition groups of $\hP_K^{(p)}$ with a nontrivial $p$-component may be
characterized as the maximal finite nilpotent subgroups of $\hP_K^{(p)}$ with a nontrivial $p$-component.
\end{enumerate}
\end{lemma}

\begin{proof}From the same argument used to prove item (i) of Lemma~\ref{absolute2},
it follows that $\wh{H}^{(p)}$ is a torsion free group.

In order to prove (ii), let us consider the Galois unramified covering $S_K\ra S$ associated to the
normal finite index subgroup $K$ of $\Pi$ and let $S_H\ra S_K$ be the $p$-covering
associated to the subgroup $H$ of $K$. As in the proof of (i) Lemma~\ref{relative}, the $p$-adic
solenoid $\mS_K^{(p)}$  of $S_K$, i.e.\ the inverse limit of all unramified $p$-coverings of $S_K$, identifies
with the $p$-adic solenoid $\mS_H^{(p)}$ of $S_H$. Since $H$ is a free product
of surface groups, it is $p$-good. Therefore, as in the proof of Lemma~\ref{relative}, there hold
$H^k(\mS_K^{(p)},\Z/p)=H^k(\mS_H^{(p)},\Z/p)=\{0\}$, for $k> 0$, and $H^0(\mS_K^{(p)},\Z/p)=\Z/p$.
Then, the proof proceeds exactly as for Lemma~\ref{relative}.
\end{proof}

\begin{definition}\label{sigma-relative}
For an element $a\in\hP_K^{(p)}$, let $\nu_K(a)$ be the minimal natural number such that there holds
$a^{\nu_K(a)}\in K^{(p)}$. For $\sg=\{\gm_1,\ldots,\gm_h\}\subset\hP_K^{(p)}$ an algebraically simple
subset, let then $\sg_{K,n}^{\hP_K^{(p)}}$, for $n\in\N^+$, be the closed normal subgroup of
$\hP_K^{(p)}$ generated by the elements
$\gm_1^{n\cdot\nu_K(\gm_1)},\ldots,\gm_h^{n\cdot\nu_K(\gm_h)}$.
\end{definition}

We then have the following generalization of Theorem~\ref{magnus}:

\begin{theorem}\label{magnus2}Given an algebraically simple subset
$\sg=\{\gm_1,\ldots,\gm_h\}\subset\hP_K^{(p)}$,
let $y\in\hP_K^{(p)}$ be an element such that $y^{\nu_K(y)}$ is a generator of the pro-$p$ group
$K^{(p)}$ and, for every $n=p^t$, with $t\in\N^+$, there exists a $k_n\in\N^+$ with the property
that there holds $y^{k_n}\in\sg_{K,n}^{\hP_K^{(p)}}$. Then, for some $s\in\ZZ\ssm\{0\}$ and
$i\in\{1,\ldots,h\}$, the element $y$ is conjugated to $\gm_i^s$.
\end{theorem}

\begin{proof}It is not restrictive to assume that $\sg=\{\gm_1,\ldots,\gm_h\}$ is a simple subset of $\Pi$.
Let us then denote by $\sg_{K,n}^{\Pi}$ the normal subgroup of $\Pi$ generated by the elements:
$$\gm_1^{n\cdot\nu_K(\gm_1)},\ldots,\gm_h^{n\cdot\nu_K(\gm_h)}.$$

The quotient group $\Pi/\sg_{K,n}^{\Pi}$ is a nodal Fuchsian group whose decomposition groups are
the conjugacy classes of the subgroups generated by the images of the elements in $\sg$.

By the same argument used in the proof of Theorem~\ref{magnus}, the group $\Pi/\sg_{K,n}^{\Pi}$
contains a normal, torsion free subgroup, which is contained as a subgroup of index a power of $p$ 
in the image of the subgroup $K$. Its relative pro-$p$ completion with respect to the image of $K$ is
then exactly the quotient group $\hP_K^{(p)}/\sg_{K,n}^{\hP_K^{(p)}}$.

For an element $a\in \Pi_K^{(p)}$, let us denote by $\bar{a}$ its image in the quotient group
$\hP_K^{(p)}/\sg_{K,n}^{\hP_K^{(p)}}$.

Since the natural epimorphism $\wh{K}^{(p)}\tura H_1(K,\Z/p)$ factors through the quotient group
$\wh{K}^{(p)}/\sg_{K,n}^{\hP_K^{(p)}}$ and, by hypotheses, the lift $y^{\nu_K(y)}$ is a generator of
the pro-$p$ group $\wh{K}^{(p)}$, the image $\bar{y}$ of $y$ then has finite order divisible by $p$.

From Lemma~\ref{relative2}, it then follows that, for all $n=p^t$, there holds
$\bar{y}\in\langle\bar{\gm}_i\rangle^{\bar{x}}$, for some $i\in\{1,\ldots,h\}$ and some
$\bar{x}\in\hP_K^{(p)}/\sg_{K,n}^{\hP_K^{(p)}}$. By an inverse limit argument, we conclude that
there holds $y\in\langle\gm_i\rangle^x$, for some $i\in\{1,\ldots,h\}$ and some $x\in\hP_K^{(p)}$.
\end{proof}

\begin{corollary}\label{magnus3}Let $K$ be a normal finite index subgroup of $\Pi$ such that,
for every algebraically simple element $x\in\hP_K^{(p)}$, the lift $x^{\nu_K(y)}\in \wh{K}^{(p)}$ is
a generator. Given an algebraically simple subset $\sg=\{\gm_1,\ldots,\gm_h\}\subset\hP_K^{(p)}$,
let $y\in\hP_K^{(p)}$ be an algebraically simple element such that, for every $n=p^t$, with $t\in\N^+$,
there exists a $k_n\in\N^+$ with the property that there holds $y^{k_n}\in\sg_{K,n}^{\hP_K^{(p)}}$.
Then, for some $s\in\ZZ^\ast$ and $i\in\{1,\ldots,h\}$, the element $y$ is conjugated to $\gm_i^s$.
\end{corollary}

\begin{remark}\label{cofinal}In Lemma~3.10 \cite{sym}, it is defined a
characteristic finite index subgroup $K_\ell$ of $\Pi$ such that, for every simple element $x\in\Pi$,
the lift $x^{\nu_{K_\ell}(y)}\in K_\ell$ is a generator.
The definition of this group can be rephrased, group theoretically as follows. Let $K$ be a finite index
characteristic subgroup of $\Pi$ such that, for all simple (and so for all algebraically simple) elements 
$\gm\in\Pi$, there holds $\gm\notin K$. For a given integer $\ell>1$, we let then $K_\ell:=[K,K]K^\ell$.
For any algebraically simple element $x\in\hP_{K_\ell}^{(p)}$, the lift
$x^{\nu_{K_\ell}(y)}\in\wh{K}_\ell^{(p)}$ is then a generator. It follows that any normal finite index
subgroup $N$ of $\Pi$, contained in $K_\ell$, satisfies the hypothesis of Corollary~\ref{magnus3}.
\end{remark}

Corollary~\ref{magnus3} and Remark~\ref{cofinal} yield a substantial refinement of
Theorem~\ref{magnus}. We need to fix some more notations.

\begin{definition}\label{sigma-relative2}
For an open normal subgroup $K$ of $\hP$ and $a\in\hP$, let $\nu_K(a)$ be the minimal
natural number such that $a^{\nu_K(a)}\in K$.
For an algebraically simple subset $\sg=\{\gm_1,\ldots,\gm_h\}\subset\hP$,
let then $\sg_{K,n}^{\hP}$, for $n\in\N^+$, be the closed normal subgroup of $\hP$
generated by the elements $\gm_1^{n\cdot\nu_K(\gm_1)},\ldots,\gm_h^{n\cdot\nu_K(\gm_h)}$.
\end{definition}

Let $\hP$ be the profinite completion of a hyperbolic surface
group $\Pi$ and $\hP_K^{(p)}$ its relative pro-$p$ completion with respect to some normal finite index
subgroup $K$. There is then a natural epimorphism:
$$\psi_K^{(p)}\co\hP\ra\hP_K^{(p)}.$$
From an inverse limit argument, Corollary~\ref{magnus3} and Remark~\ref{cofinal}, it follows:

\begin{corollary}\label{magnus4}Let $\sg=\{\gm_1,\ldots,\gm_h\}\subset\hP$ be an algebraically simple
subset and $p$ a fixed prime. Let $y\in\hP$ be an algebraically simple element such that, for a cofinal
system of open normal subgroups $\{K\}$ of $\hP$ and every $n=p^t$, with $t\in\N^+$, there exists a
$\mu_{K,n}\in\N^+$ with the property that there holds
$\psi_K^{(p)}(y)^{\mu_{K,n}}\in\psi_K^{(p)}(\sg_{K,n}^{\hP})$.
Then, for some $s\in\ZZ^\ast$ and $i\in\{1,\ldots,h\}$, the element $y$ is conjugated to $\gm_i^s$.
\end{corollary}

The assumption that the element $y$ is algebraically simple can be dropped reformulating
Corollary~\ref{magnus4} for relative pro-nilpotent, instead of relative pro-$p$, completions.
Let us denote by $\psi_K^{nil}\co\hP\ra\hP_K^{nil}$ the natural epimorphism. It holds:

\begin{theorem}\label{magnus5}Let $\sg=\{\gm_1,\ldots,\gm_h\}\subset\hP$ be an algebraically simple
subset. Let $y\in\hP$ be an element such that, for a cofinal system of open normal subgroups $\{K\}$
of $\hP$ and every $n\in\N^+$, there exists a $\mu_{K,n}\in\N^+$ with the property that
$\psi_K^{nil}(y)^{\mu_{K,n}}\in\psi_K^{nil}(\sg_{K,n}^{\hP})$.
Then, for some $s\in\ZZ$ and $i\in\{1,\ldots,h\}$, the element $y$ is conjugated to $\gm_i^s$.
\end{theorem}

\begin{proof} Let us assume that $y\neq 1$. Since $y$ has infinite
order a Sylow $p$-subgroup $Y_p$ of $\overline{\langle y\rangle}$
is infinite for some prime $p$. Let $y_p$ be a generator of $Y_p$
and $U$  an open subgroup of $\hP$ such that $y_p$ is a generator
of the maximal pro-$p$ quotient $U^{(p)}$ of $U$. Such a $U$
exists because $Y_p$ is the intersection of all open subgroups
containing it. Let $K_U$ be a subgroup in the cofinal system $\{K\}$
contained in $U$. Then its image $\td{K}_U$ in
$U^{(p)}$ intersects $\overline{\langle y_p\rangle}$ non-trivially
and is such that $\td{K}_U\cap \overline{\langle y_p\rangle}\not\leq \Phi(\td{K}_U)$.
Therefore, since  $\td{K}_U$ is a quotient of $K_U^{(p)}$, the image
of $Y_p\cap K_U$ in $K_U^{(p)}$ is not contained
in $\Phi(K_U^{(p)})$. This means that
$\psi_{K_U}^{(p)}(y)^{\nu_{K_U}(y)}$ is a generator of $K_U^{(p)}$.
Then, the conclusion follows from Theorem~\ref{magnus2} and the usual inverse limit argument.
\end{proof}

Proceeding as in the proof of Theorem~\ref{magnus2}, we can also describe normalizers of 
algebraically simple elements in the relative pro-$p$ completion of a surface group:

\begin{proposition}\label{centrel}Let $\hP^{(p)}_K$ be the relative pro-$p$ completion of a non-abelian surface
group $\Pi$ with respect to some normal finite index subgroup $K$ and let $x$ be an algebraically simple
element of $\hP^{(p)}_K$. Then, for all $n\in\ZZ\ssm\{0\}$, there holds:
$$N_{\hP_K^{(p)}}(\langle x^n\rangle)=N_{\hP_K^{(p)}}(\langle x\rangle)=\langle x\rangle.$$
\end{proposition}

\begin{proof}We can assume that $x$ is a simple element of $\Pi$. Let then $k>0$ the smallest integer
such that $x^k\in K$. The quotient group $\Xi_h:=\Pi/\langle x^{p^{hk}}\rangle^\Pi$ is a nodal Fuchsian
group which satisfies the hypotheses of Lemma~\ref{relative2} with respect to the normal subgroup
$K_h$, image of $K$ in the quotient group $\Xi_h$. Let $x_h$ be the image of $x$ in $\Xi_h$. Then,
the cyclic group $C_h$ generated by $x_h$ is a decomposition group of $\Xi_h$ with non-trivial
$p$-component for $h>0$.

From Lemma~\ref{relative2}, it follows that in the relative pro-$p$ completion $\Xi_{K_h}^{(p)}$ of
$\Xi_h$ with respect to the subgroup $K_h$, for $h>0$ and $p^h\nmid n$, there holds:
$$N_{\Xi_{K_h}^{(p)}}(C_h)=N_{\Xi_{K_h}^{(p)}}(C_h^n)=C_h.$$
The conclusion of the proposition then follows taking the inverse limit for $h\ra\infty$.
\end{proof}

\section[The complex of profinite curves]{A linearization of the complex of profinite curves.}
\label{linearization}
Let $S_g$ be a closed orientable Riemann surface of genus $g$ and let $\{P_1,\ldots, P_{n}\}$
be a set of distinct points on $S_g$.  The Teichm\"uller modular group $\GG_{g,n}$, for
$2g-2+n>0$, is defined to be the group of isotopy classes of diffeomorphisms or, equivalently,
of homeomorphisms of the surface $S_g$ which preserve the orientation and the
given ordered set $\{P_1,\ldots, P_{n}\}$ of marked points:
$$\GG_{g,n}:=\operatorname{Diff}^+(S_g,n)/\operatorname{Diff}_0(S_g,n)
\cong\hom^+(S_g,n)/\hom_0(S_g,n),$$
where $\operatorname{Diff}_0(S_g,n)$ and $\hom_0(S_g,n)$ denote the
connected components of the identity in the respective topological groups.

Forgetting the last marked point $P_{n+1}$ induces an epimorphism of Teichm\"uller modular groups
$p_{n+1}\co\GG_{g,n+1}\ra\GG_{g,n}$.

Let $S_{g,n}$ be the differentiable surface obtained removing the
points $P_1,\ldots, P_n$ from $S_g$ and let $\Pi_{g,n}:=\pi_1(S_{g,n},P_{n+1})$.
The homomorphism $p_{n+1}$ induces a short exact sequence of Teichm\"uller modular groups,
called \emph{the Birman exact sequence}:
$$1\ra\Pi_{g,n}\sr{i}{\ra}\GG_{g,n+1}\sr{p_{n+1}}{\ra}\GG_{g,n}\ra 1.$$

The monomorphism $i\co\Pi_{g,n}\hookra\GG_{g,n+1}$ sends the isotopy class of a
$P_{n+1}$-pointed oriented closed curve $\gm$ to the isotopy class of the homeomorphism
$i(\gm)$ defined pushing the base point $P_{n+1}$ all along the path $\gm$ in the direction
given by the orientation of $\gm$. It is clear that $i(\gm)$ is isotopic to the identity for
isotopies which are allowed to move the base point $P_{n+1}$ and then that $p_{n+1}(i(\gm))=1$.

There are natural faithful representations, induced by the action of homeomorphisms on the
fundamental group of the Riemann surface $S_{g,n}$:
$$\rho_{g,n}\co\GG_{g,n}\hookra\out(\Pi_{g,n})\hspace{1cm}\mathrm{and}\hspace{1cm}
\rho'_{g,n+1}\co\GG_{g,n+1}\hookra\aut(\Pi_{g,n}).$$

Since the automorphism $\rho'_{g,n+1}(i(\gm))$ is the inner automorphism $\mathrm{inn}\,\gm$,
induced by $\gm$, for all elements $\gm\in\Pi_{g,n}$, the exactness of the Birman sequence then
follows from the exactness of the standard group-theoretical short exact sequence:
$$1\ra\Pi_{g,n}\sr{\mathrm{inn}}{\ra}\aut(\Pi_{g,n})\ra\out(\Pi_{g,n})\ra 1.$$

It also follows that the representations $\rho_{g,n}$ and $\rho'_{g,n+1}$ can be recovered,
algebraically, from the Birman exact sequence and the action by restriction of inner automorphisms
of $\GG_{g,n+1}$ on its normal subgroup $\Pi_{g,n}$.

Let us now switch to the profinite setting. Let, as usual, $\hP_{g,n}$ be the profinite completion
of the fundamental group $\Pi_{g,n}$. Since the profinite group $\hP_{g,n}$ is center-free,
there is also a natural short exact sequence:
$$1\ra\hP_{g,n}\sr{\mathrm{inn}}{\ra}\aut(\hP_{g,n})\ra\out(\hP_{g,n})\ra 1.$$

Let us mention here a fundamental result of Nikolov and Segal \cite{N-S} which asserts that any
finite index subgroup of any topologically finitely generated profinite group $G$ is open.
Since such a profinite group $G$ has also a basis of neighborhoods of the identity consisting of
open characteristic subgroups, it follows that all automorphisms of $G$ are continuous and that
$\mbox{Aut}(G)$ is a profinite group as well.

Let $\hGG_{g,n}$, for $2g-2+n>0$, be the profinite completion of the Teichm\"uller modular group.
From the universal property of the profinite completion, it follows that there are natural representations:
$$\hat{\rho}_{g,n}\co\hGG_{g,n}\ra\out(\hP_{g,n})\hspace{1cm}\mathrm{and}\hspace{1cm}
\hat{\rho}'_{g,n+1}\co\hGG_{g,n+1}\ra\aut(\hP_{g,n}).$$

Let us then recall a few definitions from \cite{boggi}.

\begin{definition}\label{geometric}For $2g-2+n>0$, let the profinite groups $\tGG_{g,n+1}$
and $\cGG_{g,n}$ be, respectively, the image of $\hat{\rho}'_{g,n+1}$ in $\mbox{Aut}(\widehat{\Pi}_{g,n})$ and 
of $\hat{\rho}_{g,n}$ in $\mbox{Out}(\widehat{\Pi}_{g,n})$. For all $n\geq 0$, there is a natural isomorphism 
$\cGG_{g,n+1}\sr{\sim}{\ra}\tGG_{g,n+1}$ (cf.\ Lemma~20 in  \cite{Hoshi}). We then call 
$\cGG_{g,n}$ \emph{the congruence completion of the Teichm\"uller group} 
or, more simply, \emph{the procongruence Teichm\"uller group}.
\end{definition}

One of the most important objects in Teichm\"uller theory is the complex of curves:

\begin{definition}\label{curve-complex} A simple closed curve (s.c.c.) $\gm$ on the Riemann
surface $S_{g,n}$ is \emph{non-peripheral} if it does not bound a disc with less than two punctures.
\emph{A multi-curve} $\sg$ on $S_{g,n}$ is a set of disjoint, non-trivial, non-peripheral
s.c.c.'s on $S_{g,n}$, such that they are two by two non-isotopic.
The complex of curves $C(S_{g,n})$ is the abstract simplicial complex whose simplices are
isotopy classes of multi-curves on $S_{g,n}$.
\end{definition}

It is easy to check that the combinatorial dimension of $C(S_{g,n})$ is
$n-4$ for $g=0$ and $3g-4+n$ for $g\geq 1$.
There is a natural simplicial action of $\GG_{g,n}$ on $C(S_{g,n})$.

In order to construct a profinite version of the complex of curves, we need to reformulate its
definition in more algebraic terms.

Let $\cL_{g,n}=C(S_{g,n})_0$, for $2g-2+n>0$, be the set of isotopy classes of
non-peripheral simple closed curves on $S_{g,n}$. Let $\Pi_{g,n}/\!\sim$ be the set of conjugacy
classes of elements of $\Pi_{g,n}$ and let
$\cP_2(\Pi_{g,n}/\!\sim)$ be the set of unordered pairs of elements of $\Pi_{g,n}/\!\sim$.

For a given $\gm\in\Pi_{g,n}$, let us denote by $\gm^{\pm 1}$ the set $\{\gm,\gm^{-1}\}$
and by $[\gm^{\pm 1}]$ its equivalence class in $\cP_2(\Pi_{g,n}/\!\sim)$. Let us then define the
natural embedding $\iota\co\cL_{g,n}\hookra\cP_2(\Pi_{g,n}/\!\sim)$, choosing, for
an element $\gm\in\cL_{g,n}$, an element $\vec{\gm}_\ast\in\Pi_{g,n}$ whose free homotopy class
contains $\gm$ and letting $\iota(\gm):=[\vec{\gm}_\ast^{\pm 1}]$.

Let $\hP_{g,n}/\!\sim$ be the set of conjugacy classes of elements
of $\hP_{g,n}$ and $\cP_2(\hP_{g,n}/\!\sim)$ the profinite set of
unordered pairs of elements of $\hP_{g,n}/\!\sim$. Since
$\Pi_{g,n}$ is conjugacy separable (cf.\ \cite{stebe})  the set
$\Pi_{g,n}/\!\sim$ embeds in the profinite set $\hP_{g,n}/\!\sim$.
So, let us define the set of \emph{non-peripheral profinite
s.c.c.'s} $\hL_{g,n}$ on $S_{g,n}$ to be the closure of the set
$\iota(\cL_{g,n})$ inside the profinite set
$\cP_2(\hP_{g,n}/\!\sim)$. When it is clear from the context, we
omit the subscripts and denote these sets simply by $\cL$ and $\hL$.

An ordering of the set $\{\alpha,\alpha^{-1}\}$ is preserved by the conjugacy action and
defines \emph{an orientation} for the associated equivalence class $[\alpha^{\pm1}]\in\hL$.

For all $k\geq 0$, there is a natural embedding of the set $C(S_{g,n})_{k-1}$ of isotopy classes of
multi-curves on $S_{g,n}$ of cardinality $k$ into the profinite set $\cP_k(\hL)$
of unordered subsets of $k$ elements of $\hL$. Let us then define the set of \emph{profinite multi-curves}
on $S_{g,n}$ as the union of the closures of the sets $C(S_{g,n})_{k-1}$ inside the profinite sets
$\cP_k(\hL)$, for all $k>0$.

Let us observe that the sets of elements of $\hP_{g,n}$ in the class of a profinite multi-curve on
$S_{g,n}$ are simple in the sense of Definition~\ref{simple}. However, in general, the class in
$\cP_k(\hL)$ of a simple subset of $k$ elements of $\hP_{g,n}$ is not a profinite multi-curve
because it may contain peripheral elements.

A \emph{simplicial profinite complex} is an abstract simplicial complex whose 
set of vertices is endowed with a profinite topology such that the sets of $k$-simplices, 
with the induced topologies, are compact and then profinite, for all $k\geq 0$. For these simplicial 
complexes, the procedure which associates to an abstract simplicial complex and an ordering of
its vertex set a simplicial set produces a simplicial profinite set.

\begin{definition}\label{geopro}(cf.\ \cite{boggi}) Let $L(\hP_{g,n})$, for $2g-2+n>0$,
be the abstract simplicial profinite complex whose simplices are the profinite multi-curves on
$S_{g,n}$. The abstract simplicial profinite complex $L(\hP_{g,n})$ is called
\emph{the complex of profinite curves on $S_{g,n}$}.
\end{definition}

For $2g-2+n>0$, there is a natural continuous action of the
procongruence Teichm\"uller group $\cGG_{g,n}$ on the complex of
profinite curves $L(\hP_{g,n})$. There are finitely many orbits of
$\cGG_{g,n}$ in $L(\hP_{g,n})_k$ each containing an element of
$C(S_{g,n})_k$, for $k\geq 0$, and, by the results of Section~4
\cite{boggi}, these orbits correspond to the possible topological
types of a surface $S_{g,n}\ssm\sg$, for $\sg$ a multi-curve on
$S_{g,n}$.

The main result of \S 4 \cite{boggi} (cf.\ Theorem~4.2) was that, for all $k\geq 0$, the profinite set
$L(\hP_{g,n})_k$ is the $\cGG_{g,n}$-completion of the discrete $\GG_{g,n}$-set $C(S_{g,n})_k$,
i.e.\ there is a natural continuous isomorphism of $\cGG_{g,n}$-sets:
$$L(\hP_{g,n})_k\cong\ilim_{\ld\in\Ld}\,C(S_{g,n})_k\left/\GG^\ld\right.,$$
where $\{\GG^\ld\}_{\ld\in\Ld}$ is a tower of finite index normal
subgroup of $\GG_{g,n}$ which forms a fundamental system of
neighborhoods of the identity for the congruence topology.

The set $\cL$ of isotopy classes of non-peripheral s.c.c.'s on $S_{g,n}$ parametrizes the set of Dehn
twists of $\GG_{g,n}$, which is the standard set of generators for this group. In other words, the
assignment $\gm\mapsto\tau_\gm$, for $\gm\in\cL$, defines an embedding
$d\co\cL\hookra\GG_{g,n}$, for $2g-2+n>0$.

The set $\{\tau_\gm\}_{\gm\in\cL}$ of all Dehn twists of
$\GG_{g,n}$ is closed under conjugation and falls in a finite set
of conjugacy classes which are in bijective correspondence with
the possible topological types of the Riemann surface
$S_{g,n}\ssm\gm$. So, let us define, for the congruence completion
$\cGG_{g,n}$, the set of \emph{profinite Dehn twists} to be the
closure of the image of the set $\{\tau_\gm\}_{\gm\in\cL}$ inside
$\cGG_{g,n}$. This is the same as the union of the conjugacy
classes in $\cGG_{g,n}$ of the images of the Dehn twists of
$\GG_{g,n}$.

There is a natural $\GG_{g,n}$-equivariant map
$d_k\co\cL\ra\cGG_{g,n}$, defined by the assignment
$\gm\mapsto\tau_\gm^k$, where $\GG_{g,n}$ acts by conjugation on
$\cGG_{g,n}$. From the universal property of the
$\cGG_{g,n}$-completion and Theorem~4.2 \cite{boggi}, it then
follows that the map $d_k$ extends to a continuous
$\cGG_{g,n}$-equivariant map
$\hat{d}_k\co\hL\ra\cGG_{g,n}$, whose image is the set
of $k$-th powers of profinite Dehn twists.

In particular, a profinite s.c.c.\ $\gm\in\hL$ determines
a profinite Dehn twist, which we denote by $\tau_\gm$, in the
procongruence Teichm\"uller group $\cGG_{g,n}$. The main result of
\S 5 of \cite{boggi} (cf.\ Theorem~5.1) is that this provides a
parametrization of the set of profinite Dehn twists of the
procongruence Teichm\"uller group:

\begin{theorem}\label{parametrize}For $2g-2+n>0$ and any $k\in\ZZ\ssm{0}$, there is a natural
injective map $\hat{d}_k\co\hL\hookra\cGG_{g,n}$ which assigns to a
profinite s.c.c.\ $\gm\in\hL$ the $k$-th power of the profinite Dehn twist  $\tau_\gm$.
\end{theorem}

The purpose of this section is to show how the complex of profinite curves $L(\hP_{g,n})$ can be
"linearized" and then extract, as a consequence of this process, both a new proof and a generalization 
of Theorem~5.1 \cite{boggi} to profinite multitwists. This linearization result can also be considered a 
generalization to the profinite case of Theorem~5.1 \cite{scc}.

In order to make this statement more precise, we need to introduce more notations and definitions.
Let $K$ be an open normal subgroup of $\hP_{g,n}$ and let $p_K\co S_K\ra S_{g,n}$ be the
associated normal unramified covering of Riemann surfaces with covering transformation group
$G_K:=\hP_{g,n}/K$. Let then $\ol{S}_K$ be the closed Riemann surface obtained from $S_K$
filling in its punctures and, for a commutative unitary ring of coefficients $A$, let $H_1(\ol{S}_K,A)$
be its first homology group. There is a natural map $\psi_K\co K\ra H_1(\ol{S}_K,A)$.

\begin{definition}\label{image}
For a given $\gm\in\hL$, let us denote by the same letter an element of the profinite group
$\hP_{g,n}$ in the class of the given profinite s.c.c.. Let $\nu_K(\gm)$ be the smallest positive
integer such that $\gm^{\nu_K(\gm)}\in K$. For a profinite multi-curve $\sg\in L(\hP_{g,n})$, let us
also denote by $\sg$ a simple subset of $\hP_{g,n}$ in the class of the given multi-curve. Let then
$V_{K,\sg}$ be the primitive $A$-submodule of $H_1(\ol{S}_K,A)$ generated by the $G_K$-orbit of
the subset $\{\psi_K(\gm^{\nu_K(\gm)})\}_{\gm\in\sg}$.
\end{definition}

For $\sg=\{\gm_1,\ldots,\gm_h\}\subset\hP_{g,n}$ a simple subset, this is the same as the
$A$-submodule generated by the image $\psi_K(\sg_{K,1}^{\hP_{g,n}})$ in the
homology group $H_1(\ol{S}_K,A)$ (cf.\ Theorem~\ref{magnus2}).

Let then $\mathrm{Gr}(H_1(\ol{S}_K,A))$ be the absolute Grassmanian of primitive
$A$-submodules of the homology group $H_1(\ol{S}_K,A)$, that is to say the disjoint union of the
Grassmanians of primitive, $k$-dimensional, $A$-submodules of the homology group
$H_1(\ol{S}_K,A)$, for all $1\leq k\leq\mathrm{rank}\,H_1(\ol{S}_K,A)$.

For $A_p$ equal to the ring of $p$-adic integers $\Z_p$ or the finite field
$\F_p$, the absolute Grassmanian $\mathrm{Gr}(H_1(\ol{S}_K,A_p))$ has a natural structure of
profinite space, while, for $A_p=\Q_p$, it is a locally compact totally disconnected Hausdorff space.
In all cases, for $\sg\in L(\hP_{g,n})$, the assignment $\sg\mapsto V_{K,\sg}$ defines
a natural continuous $\cGG_{g,n}$-equivariant map:
$$\Psi_{K,p}\co L(\hP_{g,n})\lra\mathrm{Gr}(H_1(\ol{S}_K,A_p)).$$

\begin{theorem}\label{embedding}For $p>0$ a prime number, let $A_p=\F_p$, $\Z_p$ or $\Q_p$.
For $2g-2+n>0$, there is a natural continuous $\cGG_{g,n}$-equivariant injective map:
$$\widehat{\Psi}_p:=\prod_{K\unlhd\hP_{g,n}}\Psi_{K,p}\co L(\hP_{g,n})\hookra
\prod_{K\unlhd\hP_{g,n}}\mathrm{Gr}(H_1(\ol{S}_K,A_p)),$$
where $\{K\}$ is a cofinal system of open normal subgroups of
the profinite group $\hP_{g,n}$.
\end{theorem}
\begin{proof}For $A_p=\Z_p$, the submodule $V_{K,\sg}$ of $H_1(\ol{S}_K,\Z_p)$ is primitive.
Therefore, it is enough to prove the theorem for $A_p=\F_p$.

Let $u_1,\ldots,u_n\in\Pi_{g,n}$ be simple loops around the punctures on the surface $S_{g,n}$
labeled by the points $P_1,\ldots,P_n$, respectively.

For $K$ an open normal subgroup of $\hP_{g,n}$,
a simple set of elements $\sg=\{\gm_1,\ldots,\gm_h\}$ and $s\in\N^+$, let us denote by
$\td{\sg}_{K,s}^{\hP_{g,n}}$ the closed normal subgroup generated by the set of elements:
$$\gm_1^{s\cdot\nu_K(\gm_1)},\ldots,\gm_h^{s\cdot\nu_K(\gm_h)},
u_1^{s\cdot\nu_K(u_1)},\ldots,u_n^{s\cdot\nu_K(u_n)}.$$

For $\sg=\{\gm_1,\ldots,\gm_h\}\in L(\hP_{g,n})$, let $\td{V}_{K,\sg}$ be the subspace of the
$\F_p$-vector space $K_p^{ab}:=H_1(K,\F_p)$ which is the image of the normal subgroup
$\td{\sg}_{K,s}^{\hP_{g,n}}$ of $K$ by the natural epimorphism $K\ra K_p^{ab}$. The assignment
$\gm\mapsto \td{V}_{K,\sg}$ then defines a natural continuous $\cGG_{g,n}$-equivariant map:
$$\td{\Psi}_{K,p}\co L(\hP_{g,n})\lra\mathrm{Gr}(K_p^{ab}).$$

For $\xi=\{\delta_1,\ldots,\delta_h\}\in L(\hP_{g,n})$, let us also denote by $\td{\xi}_{K,s}^{\hP_{g,n}}$
the closed normal subgroup generated by the set of elements:
$$\delta_1^{s\cdot\nu_K(\gm_1)},\ldots,\delta_h^{s\cdot\nu_K(\gm_h)},
u_1^{s\cdot\nu_K(u_1)},\ldots,u_n^{s\cdot\nu_K(u_n)}.$$

Let us show that, if $\sg\neq \xi\in L(\hP_{g,n})$, there is an open normal subgroup $K$
of $\hP_{g,n}$ such that there holds $\td{\Psi}_{K,p}(\sg)\neq\td{\Psi}_{K,p}(\xi)$.
This obviously implies that there holds as well $\Psi_{K,p}(\sg)\neq\Psi_{K,p}(\xi)$, proving
Theorem~\ref{embedding}.

A consequence of Theorem~4.2 \cite{boggi} is also that not every element of $\sg$ is conjugated
to a $\ZZ$-power of an element in $\xi$. Since no element of $\sg$ is conjugated to $u_k$, for
$k=1,\ldots,n$, from Corollary~\ref{magnus4}, it then follows that
there is an open normal subgroup $H$ of $\hP_{g,n}$ and an $s=p^t$, for $t\in\N$, such that
the images of the subgroups $\td{\sg}_{H,s}^{\hP_{g,n}}$ and $\td{\xi}_{H,s}^{\hP_{g,n}}$
in the maximal pro-$p$ quotient $H^{(p)}$ of the profinite group $H$ are distinct. Moreover,
by Remark~\ref{cofinal}, we can assume that all elements $\gm_i^{\nu_K(\gm_i)}$
and $\delta_j^{\nu_K(\delta_j)}$, for $i,j=1,\ldots, h$, are generators in $H$.

Let then $L$ be an open normal subgroup of $\hP_{g,n}$ contained in $H$ and of index
a power of $p$ in $H$ such that there holds $\nu_L(x)=s\nu_H(x)$, for $x=\gm_i,\delta_j$ or $u_k$, for
$i,j=1,\ldots, h$ and $k=1,\ldots,n$. Then, the images of the subgroups
$\td{\sg}_{L,1}^{\hP_{g,n}}=\td{\sg}_{H,s}^{\hP_{g,n}}$ and
$\td{\xi}_{L,1}^{\hP_{g,n}}=\td{\xi}_{H,s}^{\hP_{g,n}}$ in the maximal pro-$p$ quotient $L^{(p)}$
of $L$ are also distinct. We need one more lemma:

\begin{lemma}\label{Frattini} Let $L^{(p)}$ be a pro-$p$ group and $N_1\neq N_2$ normal subgroups of
$L^{(p)}$ invariant for the action of a finite subgroup $G$ of $\out(L^{(p)})$. Then, there exists an open
normal subgroup $U$ of $L^{(p)}$, containing $N_1,N_2$ and invariant for the action of $G$, such that
$N_1\Phi(U)\neq N_2\Phi(U)$, where, for a given group $H$, we denote by $\Phi(H)$ its Frattini subgroup.
\end{lemma}

\begin{proof}Note that $N_1\subsetneq N_1N_2$. Hence, it suffices to prove the existence of
an open normal subgroup $U$ of $L^{(p)}$ containing $N_1N_2$ and invariant for the action of $G$
such that there holds $N_1\Phi(U)\neq N_1N_2\Phi(U)$.

Let $\{V_\upsilon\}_{\upsilon\in\Upsilon}$ be the set of all open normal subgroups of $L^{(p)}$ containing
$N_1N_2$ and invariant for the action of $G$, it then holds
$\bigcap_{\upsilon\in\Upsilon} V_\upsilon=N_1N_2$ and
$\bigcap_{\upsilon\in\Upsilon} \Phi(V_\upsilon)=\Phi(N_1N_2)$.

If $N_1\Phi(V_\upsilon)=N_1N_2\Phi(V_\upsilon)$, for all $\upsilon\in\Upsilon$,
then there holds as well:
$$N_1=N_1\Phi(N_1)=\bigcap_{\upsilon\in\Upsilon} N_1\Phi(V_\upsilon)=
\bigcap_{\upsilon\in\Upsilon}N_1N_2\Phi(V_\upsilon)=N_1N_2\Phi(N_1N_2)=N_1N_2.$$
Therefore, since $N_1\neq N_1N_2$, there holds $N_1\Phi(V_\upsilon)\neq N_1N_2\Phi(V_\upsilon)$,
for some $\upsilon\in\Upsilon$.
\end{proof}

In order to complete the proof, we apply Lemma~\ref{Frattini} with $G=\hP_{g,n}/L$,
$N_1=\psi_K^{(p)}(\td{\sg}_{L,1}^{\hP_{g,n}})$ and $N_2=\psi_K^{(p)}(\td{\xi}_{L,1}^{\hP_{g,n}})$.
Then, let $U$ be an open normal subgroup of $L^{(p)}$ as in the statement of Lemma~\ref{Frattini},
let $U'$ be its inverse image in $\hP_{g,n}$, which is also a normal subgroup, and let $K:=L\cap U'$.
Since $K$ contains both $\td{\sg}_{L,1}^{\hP_{g,n}}$ and $\td{\xi}_{L,1}^{\hP_{g,n}}$,
there holds $\td{\Psi}_{K,p}(\sg)\neq\td{\Psi}_{K,p}(\xi)$.
\end{proof}

\section[Centralizers of profinite multitwists]{Centralizers of profinite multitwists in the procongruence Teichm\"uller group}\label{centralizers}
The results of the previous section have interesting implications for the combinatorial structure 
of the procongruence Teichm\"uller group. We will improve the results both of \cite{boggi} and of \cite{H-M2}.

As a first application of Theorem~\ref{embedding}, it is now possible to give a partial parametrization also 
to profinite multitwists, i.e.\ products of powers of commuting profinite Dehn twists. 
Let us observe that, for $\sg=\{\gm_0,\ldots,\gm_k\}$ a profinite multi-curve on $S_{g,n}$, 
the set $\{\tau_{\gm_0},\ldots,\tau_{\gm_k}\}$ is a set of commuting profinite Dehn twists.

\begin{theorem}\label{multitwists}For $2g-2+n>0$, let $\sg=\{\gm_1,\ldots,\gm_s\}$ and
$\sg'=\{\delta_1,\ldots,\delta_t\}$ be two profinite multi-curves on $S_{g,n}$.
Suppose that, there is an identity in $\cGG_{g,n}$:
$$\tau_{\gm_1}^{h_1}\tau_{\gm_2}^{h_2}\cdot\ldots\cdot\tau_{\gm_s}^{h_s}=
\tau_{\delta_1}^{k_1}\tau_{\delta_2}^{k_2}\cdot\ldots\cdot\tau_{\delta_t}^{k_t},$$
for $h_i\in m_\sg\cdot\N^+$ and $k_j\in m_{\sg'}\cdot\N^+$, with $m_\sg,m_{\sg'}\in\ZZ^\ast$.
Then, there hold:
\begin{enumerate}
\item $t=s$;
\item there is a permutation $\phi\in\Sigma_s$ such that $\delta_i=\gm_{\phi(i)}$ and $k_i=h_{\phi(i)}$,
for $i=1,\ldots,s$.
\end{enumerate}
\end{theorem}

\begin{proof}If $\sg\neq\sg'$, by Theorem~\ref{embedding}, there is an open characteristic
subgroup $K$ of $\hP_{g,n}$ such that there holds
$\Psi_{K,p}(\sg)\neq\Psi_{K,p}(\sg')$ in the $\Q_p$-vector space $H_1(\ol{S}_K,\Q_p)$.

Let $\cGG^K$ be the geometric level associated to $K$, i.e.\ the kernel
of the natural representation $\cGG_{g,n}\ra\out(\hP_{g,n}/K)$. Then (cf.\ \S 2 \cite{boggi}), there is a
natural representation $\rho_{K,(p)}\co\cGG^K\ra Z_{\Sp(H_1(\ol{S}_K,\Q_p))}(G_K)$.

Let $r\in\N^+$ be such that, for every profinite Dehn twist $\tau_\gm\in\cGG_{g,n}$, there holds
$\tau_\gm^r\in\cGG^K$. From the results of \S 5 in \cite{boggi}, it follows that it is possible to
recover the subspaces $\Psi_{K,p}(\sg)$ and
$\Psi_{K,p}(\sg')$ as the \emph{cores} (cf.\ remarks preceding Lemma~5.11 \cite{boggi}) of the
symmetric bilinear forms on $H_1(\ol{S}_K,\Q_p)$ associated to the multi-transvections
$\rho_{K,(p)}(\tau_{\gm_1}^{r\cdot h_1}\ldots\tau_{\gm_s}^{r\cdot h_s})$ and
$\rho_{K,(p)}(\tau_{\delta_1}^{r\cdot k_1}\ldots\tau_{\delta_t}^{r\cdot k_t})$, respectively. Therefore,
the hypotheses of the theorem imply $\sg=\sg'$, but then items $(i)$ and $(ii)$ follow immediately.
\end{proof}

An immediate consequence of Theorem~\ref{multitwists} is a description of centralizers of profinite
multitwists of the procongruence Teichm\"uller group, generalizing Theorems~D and E in \cite{H-M2}
in which this result was proved for maximal multi-curves.

\begin{corollary}\label{centralizers multitwists}\begin{enumerate}
\item For $2g-2+n>0$, let $\sg=\{\gm_1,\ldots,\gm_s\}$ be a profinite multi-curve on $S_{g,n}$ and
$(h_1,\ldots,h_k)\in(m_\sg\cdot\N^+)^k$ a multi-index, with $m_\sg\in\ZZ^\ast$. Then, there holds:
$$Z_{\cGG_{g,n}}(\tau_{\gm_1}^{h_1}\cdot\ldots\cdot\tau_{\gm_k}^{h_k})
= N_{\cGG_{g,n}}(\langle \tau_{\gm_1}^{h_1}\cdot\ldots\cdot\tau_{\gm_k}^{h_k}\rangle)
=N_{\cGG_{g,n}}(\langle \tau_{\gm_1},\ldots,\tau_{\gm_k}\rangle).$$

\item Let us assume that $\sg=\{\gm_1,\ldots,\gm_s\}$ is a multi-curve on $S_{g,n}$
such that there holds $S_{g,n}\ssm\{\gm_1,\ldots,\gm_k\}\cong S_{g_1,n_1}\amalg\ldots\amalg S_{g_h,n_h}$.
Then, the centralizer in the procongruence Teichm\"uller modular group $\cGG_{g,n}$ of the
multitwist $\tau_{\gm_1}^{h_1}\cdot\ldots\cdot\tau_{\gm_k}^{h_k}$ is the closure, inside
$\cGG_{g,n}$, of the stabilizer $\GG_\sg<\GG_{g,n}$. Therefore, it is described by the exact sequences:
$$\begin{array}{c}
1\ra\cGG_{\vec{\sg}}\ra Z_{\cGG_{g,n}}(\tau_{\gm_1}^{h_1}\cdot\ldots\cdot\tau_{\gm_k}^{h_k})
\ra\mathrm{Sym}^{\pm}(\sg),\\
\\
1\ra\bigoplus\limits_{i=1}^k\ZZ\cdot\tau_{\gm_i}\ra\cGG_{\vec{\sg}}
\ra\cGG_{g_1,n_1}\times\dots\times\cGG_{g_h,n_h}\ra 1,
\end{array}$$
where $\mathrm{Sym}^{\pm}(\sg)$ is the group of signed permutations on the set $\sg$.
 \end{enumerate}
\end{corollary}

\begin{proof}As we already observed, for $f\in\cGG_{g,n}$, there holds the identity:
$$f\cdot(\tau_{\gm_1}^{h_1}\cdot\ldots\cdot\tau_{\gm_k}^{h_k})\cdot f^{-1}=
\tau_{f(\gm_1)}^{h_1}\cdot\ldots\cdot\tau_{f(\gm_k)}^{h_k}.$$
The conclusion then follows from Theorem~\ref{multitwists}, Corollary~6.4 and
Theorem~6.6 in \cite{boggi}.
\end{proof}

\section{Galois actions on hyperbolic curves.}\label{faith}
Let $C$ be a hyperbolic curve defined over a number field $\K$. Let us fix an embedding
$\K\subset\ol{\Q}$ and a $\ol{\Q}$-valued point $\tilde{\xi}\in C$. The structural morphism
$C\ra\Spec(\K)$ induces a short exact sequence of algebraic fundamental groups:
$$1\ra\pi_1(C\times_\K\ol{\Q},\tilde{\xi})\ra\pi_1(C,\tilde{\xi})\ra G_\K\ra 1,$$
where $G_\K$ is the absolute Galois group and the group $\pi_1(C\times_\K\ol{\Q})$ is isomorphic
to the profinite completion of a hyperbolic surface group.
Associated to the above short exact sequence, is the outer Galois representation:
$$\rho_C\co G_\K\ra\out(\pi_1(C\times_\K\ol{\Q},\tilde{\xi})).$$

By Theorem~2.2 \cite{Matsu}, Corollary~6.3 \cite{H-M} and Theorem~7.7 \cite{boggi}, the representation
$\rho_C$ is faithful. In this section, as an application of the restricted Magnus property for
profinite surface groups, we are going to prove some refinements of these results.

\begin{theorem}\label{faithfulness}Let $C$ be a smooth $n$-punctured, genus $g$ curve, defined
over a number field $\K$. For $2g-2+n>0$ and $3g-3+n>0$, the faithful outer Galois
representation $\rho_C$ induces a faithful representation:
$$\omega_{g,n}\co G_\K\hookra\aut(L(\hP_{g,n})).$$
\end{theorem}

\begin{proof}By the definition of profinite simple closed curves, it is clear that the representation
$\rho_C$ induces a natural representation $G_\K\ra\aut(\hL_{g,n})$, which induces
the natural representation $\omega_{g,n}$. The faithfulness of the representation $\omega_{g,n}$
then follows from Theorem~7.2 and Corollary~7.6 of \cite{boggi}.
\end{proof}

Let $\{C^\ld\}_{\ld\in\Ld}$ be the tower of Galois \'etale connected covering of $C$ associated to
characteristic subgroups of $\hP:=\pi_1(C\times_\K\ol{\Q},\tilde{\xi})$ and let us denote by $G^\ld$ the
covering transformation group of the covering $C^\ld\ra C$. For $\ld\in\Ld$, let us also denote
by $\ol{C}^\ld$ the smooth projective curve obtained from $C$ filling in its punctures.

For a $G$-vector space $V$, let us denote by $\mathrm{Gr}_G(V)$ the absolute Grassmanian of
$G$-invariant subspaces of $V$. The outer Galois representation $\rho_C$ then induces, for every 
$\ld\in\Ld$, a natural representation
$G_\K\ra\aut\,(\mathrm{Gr}_{G^\ld}(H_{\acute{e}t}^1(\ol{C}^\ld,\Q_\ell)))$.

Let $H_{\acute{e}t}^1(\ol{C}^\infty,\Q_\ell):=\dlim_{\ld\in\Ld}H_{\acute{e}t}^1(\ol{C}^\ld,\Q_\ell)$.
This space is endowed with a natural continuous action of the profinite group $\hP$.
Let then $\mathrm{Gr}_{\hP}(H_{\acute{e}t}^1(\ol{C}^\infty,\Q_\ell))$ be the absolute Grassmanian of
$\hP$-invariant subspaces of $H_{\acute{e}t}^1(\ol{C}^\infty,\Q_\ell)$.

From Theorem~\ref{faithfulness} and Theorem~\ref{embedding}, it follows:

\begin{corollary}\label{galois}Let $C$ be a hyperbolic curve defined over a number field $\K$ with
non-trivial moduli space. The associated outer Galois representation
$\rho_C$ then induces a natural faithful representation:
$$G_\K\hookra\aut\,(\mathrm{Gr}_{\hP}(H_{\acute{e}t}^1(\ol{C}^\infty,\Q_\ell))).$$
\end{corollary}

\vspace{1cm}

\noindent Marco Boggi,\\ Departamento de Matem\'atica, UFMG, \\
Av. Ant\^onio Carlos, 6627 - Caixa Postal 702 \\ 
CEP 31270-901 - Belo Horizonte - MG, Brasil.
\\
E--mail:\,\,\, marco.boggi@gmail.com

\vspace{1cm}

\noindent Pavel Zalesskii,\\ Departamento de Matem\'atica, Universidade de Bras\'{\i}lia, \\
70910-900, Bras\'{\i}lia-DF, Brasil.
\\
E--mail:\,\,\, pz@mat.unb.br


\begin{thebibliography}{PP}


 \bibitem{A-M} M. Artin, B. Mazur. \textsl{Etale Homotopy}. Springer Lecture Notes in Mathematics 
 {\bf 100} (1969). 
 


\bibitem{sym} M.\ Boggi. \textsl{Galois covers of moduli spaces of curves and loci of
curves with symmetries}. Geom.\ Dedicata n.\ {\bf 168} (2014), 113--142.

\bibitem{boggi} M.\ Boggi. \textsl{On the procongruence completion of the Teichm\"uller
modular group}. Trans.\ Amer.\ Math.\ Soc.\ {\bf 366} (2014), no.\ 10, 5185--5221.


\bibitem{scc} M.~Boggi, P.A.\ Zalesskii.
\textsl{Characterizing closed curves on Riemann surfaces 
via homology groups of coverings.} IMRN Vol. 2016, no. 00, pp. 1--22.

\bibitem{B-K-Z}O.\ Bogopolski, E.\ Kudryavtseva, H.\ Zieschang \textsl{Simple curves on surfaces
and an analog of a theorem of Magnus for surface groups}. Math.\ Zeitschrift n.\ {\bf 247} (2004),
595--609.

\bibitem{Bogo}O.\ Bogopolski. \textsl{A surface groups analogue of a theorem of Magnus}.
From \emph{Geometric methods in group theory}. Contemp.\ Math.\ n.\ {\bf 372}, Amer. Math. Soc.,
Providence, RI (2005), 59--69.

\bibitem{BCR}M.R. Bridson, M.D.E. Conder, A.W. Reid \textsl{Determining Fuchsian groups by their finite quotients}. 
\texttt{arXiv:1401.3645v2} (2015). To appear on \emph{Israel Mathematical Journal}.

\bibitem{Brown}K.S.\ Brown. \textsl{Cohomology of groups}. Graduate Texts in Mathematics
{\bf 87}, Springer-Verlag (1982).



\bibitem{EHKZ}A. Engler, D. Haran, D. Kochloukova, P.A. Zalesskii. \textsl{Normal subgroups of profinite groups of finite cohomological dimension}. 
J. London Math. Soc. (2) {\bf 69} (2004), no. 2, 317--332. 




\bibitem{GJZ}F.\ Grunewald, A.\ Jaikin-Zapirain, P.A.\ Zalesskii. \textsl{Cohomological goodness
and the profinite completion of Bianchi groups}. Duke Math.\ J.\ {\bf 144} (2008), no.\ 1, 53--72.


\bibitem{HW} F.\ Haglund, D.T.\ Wise. \textsl{Special cube complexes.} Geom.\ Funct.\
Anal.\ {\bf 17} (2008), no.\ 5, 1551--1620.

\bibitem{Hoshi}Y.\ Hoshi. \textsl{On Monodromically Full Points of Configuration Spaces
of Hyperbolic Curves}. In \emph{The Arithmetic of Fundamental Groups PIA 2010}. Ed. Jacob Stix. 
Contributions in Mathematical and Computational Sciences - Volume 2. Springer-Verlag Berlin Heidelberg (2012), 167--207.

\bibitem{H-M}Y.\ Hoshi, S.\ Mochizuki. \textsl{On the combinatorial anabelian geometry
of nodally nondegenerate outer representations}. Hiroshima Math. J.
Volume {\bf 41}, Number 3 (2011), 275--342.

\bibitem{H-M2}Y.\ Hoshi, S.\ Mochizuki. \textsl{Topics surrounding the combinatorial anabelian geometry
of hyperbolic curves II: tripods and combinatorial cuspidalization}. RIMS pre-print {\bf 1762} (2014).






\bibitem{M-K-S}W.\ Magnus, A.\ Karrass, D.\ Solitar \textsl{Combinatorial
group theory.} Second Revised Edition. Dover Publications, Inc.\ New York (1976).

\bibitem{Mag}W.\ Magnus. \textsl{\"Uber diskontinuierliche Gruppen mit einer definierenden
Relation. (Der Freiheitssatz.)} J.\ Reine Angew.\ Math.\ n.\ {\bf 163} (1930), 141--165.

\bibitem{Matsu}M.\ Matsumoto. \textsl{Galois representations on profinite braid groups on curves}.
J.\ Reine Angew.\ Math.\ n.\ {\bf 474} (1996), 169--219.


\bibitem{Mochi}S.\ Mochizuki. \textsl{Absolute anabelian cuspidalizations of proper
hyperbolic curves}. J.\ of Math.\ Kyoto Univ.\ n.\ {\bf 47} (2007), 451--539.

\bibitem{N-S}N. Nikolov, D. Segal. \textsl{On finitely generated profinite groups, I: strong completeness and uniform bounds}. 
Annals of Math., Vol. {\bf 165} (2006), 171--238.

\bibitem{R-Z}L.\ Ribes, P.A.\ Zalesskii. \textsl{Profinite groups}. Erg.\ der Math.\
und ihrer Grenz.\ 3.\ Folge {\bf 40}, Springer-Verlag (2000).


\bibitem{Serre}J.P. Serre. \textsl{Cohomologie galoisienne. Cinqui{\'e}me {\'e}dition, r{\'e}vis{\'e}e et
compl{\'e}t{\'e}e}. Lectures Notes in Mathematics n.{\bf 5}, Springer-Verlag (1997).

\bibitem{Scheiderer} C.\ Scheiderer. \textsl{Farrell cohomology and Brown theorems for profinite
groups}. Manuscripta Math.\ n.\ {\bf 91} (1996), 247--281.

\bibitem{stebe} P.F.\ Stebe. \textsl{Conjugacy separability of certain free products
with amalgamation}. Trans.\ Amer.\ Math.\ Soc.\ {\bf 156} (1971), 119--129.


\bibitem{Wise-qc-h} D.T.\ Wise. \textsl{The structure of groups with a quasiconvex hierarchy}.
Pre-print.

\bibitem{ZM} P.A.\ Zalesskii, O.\ Melnikov. \textsl{Subgroups of profinite groups acting on trees.}
Math.\ USSR Sbornik, \textbf{63} (1989), 405--424.




\end{thebibliography}
\end{document}